\documentclass{amsart}
\usepackage{amsfonts}
\usepackage{amsmath}
\usepackage{amssymb}
\usepackage[applemac]{inputenc}
\usepackage{url}
\usepackage{hyperref}
\usepackage{color}
\usepackage{lipsum}
\usepackage[title]{appendix}
\usepackage{mathtools}
\usepackage{ulem}

\usepackage{graphicx}

\def\C{\mathbb{C}}
\def\R{\mathbb{R}}

\def\Z{\mathbb{Z}}

\def\F{\mathcal{F}}

\def\F{\mathcal {F}}

\def\T{\mathbb{T}}

\def\Heis{{\sf{Heis}} }
\def\Der{{\sf{Der}}}

\def\heis{{\mathfrak{heis}} }

\def\isom{\mathsf{Isom}}

\def\Isom{{\sf Isom}}

\def\Aff{\sf{Aff}}
\def\Sim{\sf{Sim}}

\def\Euc{\mathsf{Euc}}

\newcommand\ip{{\langle\cdot \,,\cdot \rangle}}

\def\Aut{{\sf{Aut}}}

\def\GL{{\sf{GL}}}
\def\O{{\sf{O}}}
\def\SO{{\sf{SO}}}

\def\Span{{\sf{Span}} }

\def\ad{\mathsf{ad}}

\def\Ad{\mathsf{Ad}}

\def\Mink{{\sf{Mink}} }

\def\AdS{{\sf{AdS}} }

\def\det{{\sf{det}}}

\def\p{{\mathfrak{p}}}
\def\k{{\mathfrak{k}}}

\def\a{{\mathfrak{a}}}
\def\g{{\mathfrak{g}}}

\def\i{{\mathfrak{i}}}
\def\z{{\mathfrak{z}}}
\def\z{{\mathfrak{z}}}
\def\q{{\mathfrak{q}}}

\def\u{{\mathfrak{u}}}

\newtheorem{theorem}{{Theorem}}[section]
\newtheorem{proposition}[theorem]{{Proposition}}%[section]
\newtheorem{isom.ext}[theorem]{{Trivial isometric extension}}%[section]
\newtheorem{definition}[theorem]{{Definition}}%[section]
\newtheorem{lemma}[theorem]{{Lemma}}%[section]
\newtheorem{cor}[theorem]{{Corollary}}%[section]
\newtheorem{fact}[theorem]{{\sc Fact}}%[section]
\newtheorem{remark}[theorem]{{Remark}}%[section]
\newtheorem{question}[theorem]{{Question}}%[section]
\newtheorem{conv}[theorem]{{Convention}}%[section]
\newtheorem{example}[theorem]{{Example}}%[section]
\newcommand\blue{\color{blue}}

\definecolor{purple}{rgb}{0.65,0.12,0.94}
\definecolor{forestgreen}{rgb}{0.4,0.64,0.13}
\hypersetup{
	colorlinks=true,
	linkcolor=blue,
	filecolor=magenta,      
	urlcolor=black,
    citecolor={cyan}
}

\begin{document}
	\date{ \today}
 \title[Topology and Dynamics of compact plane waves]{Topology and Dynamics of compact plane waves
 }
\author[M. Hanounah]{Malek Hanounah}
\address{Institut f\"ur Mathematik und Informatik \hfill\hfill\break\indent
Walther-Rathenau-Str. 47,
17489 Greifswald}
\email{malek.hanounah@uni-greifswald.de}

\author [I. Kath]{Ines Kath}
\address{Institut f\"ur Mathematik und Informatik \hfill\hfill\break\indent
Walther-Rathenau-Str. 47,
17489 Greifswald}
\email{ines.kath@uni-greifswald.de}

\author [L. Mehidi]{Lilia Mehidi}
\address{Departamento de Geometria y Topologia\hfill\break\indent
Facultad de Ciencias, Universidad de Granada, Spain
%\hfill\break\indent 
}
\email{lilia.mehidi@ugr.es}
 
\author[A. Zeghib]{Abdelghani Zeghib }
\address{UMPA, CNRS, 
%Unit\'e de Math\'ematiques Pures et Appliqu\'ees
%\hfill\break\indent
\'Ecole Normale Sup\'erieure de Lyon\hfill\break\indent
46, all\'ee d'Italie
%\hfill\break\indent
69364 LYON Cedex 07, FRANCE}
\email{abdelghani.zeghib@ens-lyon.fr 
\hfill\break\indent
\url{http://www.umpa.ens-lyon.fr/~zeghib/}}

\date{\today}

\begin{center}
\begin{abstract} 
We study compact locally homogeneous plane waves. Such a manifold is a quotient of a homogeneous plane wave $X$ by a discrete subgroup of its isometry group. This quotient is called standard if the discrete subgroup is contained in a connected subgroup of the isometry group that acts properly cocompactly on $X$. We show that compact quotients of homogeneous plane waves are ``essentially" standard; more precisely, we show that they are standard or `semi-standard'.  We find conditions which ensure that a quotient is not only semi-standard but even standard. As a consequence of these results, we obtain that the flow of the parallel lightlike vector field of a compact locally homogeneous plane wave is equicontinuous.
\end{abstract}
\end{center}
\maketitle
\tableofcontents

\section{Introduction}
 
The general theme of the present article is the study of the fundamental group and the isometry group of compact locally homogeneous Lorentzian manifolds. More precisely, the Lorentzian metrics that we consider here are plane waves (see Definition~\ref{defpw}). Compact locally homogeneous plane waves are quotients of a homogeneous plane wave by a discrete subgroup of its isometry group. So they fit into the following more general situation. Let $X = G/H$ be a homogeneous space (not necessarily endowed with a metric), and let $\Gamma \subset G$ be a discrete subgroup acting 
properly, cocompactly and freely on $X$. Then $\Gamma\backslash X$ is a manifold. It is called a compact quotient of $X$. 
This leads to the problem of describing the discrete subgroups $\Gamma \subset G$ acting properly and cocompactly on $X$.   
Of particular interest is such a description if $X$ is a semi-Riemannian homogeneous space. The Riemannian case has a long history, especially in the case of constant sectional curvature. The pseudo-Riemannian situation is comparatively less studied and involves a lot of additional difficulties.
   
In the following we give a little insight into the problems that we want to deal with and thereby review some classical results. We also present our results in brief. We start with flat compact Riemannian and Lorentzian manifolds before turning to our actual object of study, compact locally  homogeneous plane waves.

\begin{conv}
We say that a group has a virtual property P if it contains a finite index subgroup which has property P.
\end{conv}

   \subsection{Flat case} \label{Subsection: Flat case} The flat Riemannian case, that is when 
   $X$ is the Euclidean space, corresponds to the crystallographic problem handled by Bieberbach's theorem, which states that the fundamental group $\Gamma \subset \Isom(\R^n) =
   \O(n) \ltimes \R^n$ is in fact contained in $\R^n$  (up to  finite index) see \cite{bieberbach1911bewegungsgruppen}. The converse is also true, namely, any group which is torsion-free, finitely generated and virtually isomorphic to $\Z^{n}$ can be realized as the fundamental group of a compact flat Riemannian manifold \cite[Theorem 1.3]{milnor1977fundamental}.
   
   The Lorentzian flat case is much more complicated.  
   We can summarize the state of current research as follows.

  \begin{theorem}\label{Introduction-Theorem: flat case} Let $M^{n+1}$ be a connected compact flat Lorentzian manifold. Then:
     
     \begin{enumerate}
     
     \item  {\sc Completeness:}   $M$ is the quotient of the Minkowski space $\Mink^{1, n}$ by a discrete subgroup $\Gamma$
     of $\Isom (\Mink^{1, n})$ acting properly and freely.
     \item[(2)] {\sc Fundamental group:} $\Gamma$ is virtually polycyclic. More exactly, $\Gamma$ is either virtually nilpotent or virtually an abelian extension of $\Z$, i.e. virtually $\Z\ltimes \Z^n$.
     \item[(3)] {\sc Standardness:}  Up to a finite index, $\Gamma$ is a cocompact lattice in a solvable connected Lie subgroup $L$
     of $\Isom (\Mink^{1, n})$ acting simply transitively on $ \Mink^{1, n}$. (In other words, $M = \Gamma \setminus L$, where $L$ is a Lie group endowed with a complete and flat left invariant Lorentzian metric). 
     \end{enumerate}
 \end{theorem}
  Item (1) was proved by  Carri\`ere \cite{Carriere}. The first part of item (2) is proved by Goldman and  Kamishima \cite{goldman1984fundamental}, and  a classification up to abstract commensurability is achieved by Grunewald and Margulis \cite{grunewald1988transitive}. Item (3) is done by Fried, Goldman, and Kamishima  \cite[\S 1]{fried1983three} \cite{goldman1984fundamental}. For an excellent survey on this topic see \cite{carriere1_Bieberbach}.
  
\bigskip 
We are going to prove the following result (see  Section \ref{Section: Flat case}).
\begin{theorem}[{\sc Parallel fields}]\label{Introduction-Theorem: parallel fields}  If $M^{n+1}$ is a connected compact flat Lorentzian manifold, then up to taking a finite cover, $M$ has a parallel vector field $V$. \\
     $\bullet$ If $V$ is timelike, then $M$ is a flat Lorentzian torus, that is 
     $\Gamma$ consists of translations, and $V$ is a linear flow. \\
     $\bullet$ If $V$ is lightlike, then its flow is equicontinuous. \\
     $\bullet$ If $V$ is spacelike, then its dynamics is encoded in the linear part of $L$, that is, its image under the linear part projection $\Isom (\Mink^{1, n}) = \SO(1, n) \ltimes  \R^{1+n} \to \SO(1, n)$. In particular, there are examples where the flow of $V$ is equicontinuous, Anosov  or more generally partially hyperbolic.\\  
     $\bullet$ In all cases, up to a finite cover, $M$ admits a parallel field of lightlike directions, i.e. an oriented parallel lightlike line field.
     \end{theorem}
 To our best knowledge, existence of parallel vector fields was never stated in the literature. It can be deduced by carefully reading the more general results in \cite{grunewald1988transitive}. Our proof is only targeted on the existence of such a vector and is therefore simpler.

Example \ref{Example: SOL-Anosov} is an interesting example of a non-equicontinuous spacelike parallel flow.

 \begin{remark}[Kundt spacetimes]\label{remark: Kundt space- weakly plane wave}
      It is worth pointing out that by Theorem \ref{Introduction-Theorem: flat case}, any flat compact Lorentzian manifold is ``weakly'' plane wave, i.e. a plane wave in the sense of Definition \ref{defpw}, but with a parallel lightlike line field instead of a parallel lightlike vector field. This class belongs to the so-called ``weakly" Brinkmann manifolds, i.e. Lorentzian manifolds with a parallel lightlike line field. The latter class fits in a larger class of manifolds called locally Kundt spacetimes, defined as those having a codimension one lightlike geodesic foliation. These spacetimes are of a great importance in general relativity, see \cite{boucetta2022kundt} for more details on the subject. 
 \end{remark}

\begin{remark}[Constant curvature] Klingler \cite{klingler1996completude} extended Carri\`ere's result by showing that any compact Lorentzian manifold of constant curvature is complete. Namely, in curvature $-1$ it is the quotient of anti-de-Sitter space $\AdS^{1,n}$ by a discrete subgroup of $\isom(\AdS^{1,n})$. In the positive curvature, due to a classical result known as the Calabi-Markus phenomena \cite{markus-calabi}, there are no compact Lorentzian manifolds with positive constant curvature. On the other hand, compact quotients exist in negative curvature only if $n$ is even.  In dimension $3$, Goldman shows in \cite{goldman1985nonstandard} that there are non-standard (for the definition of standard see Section \ref{Section: syndetic hull, standard semi-Standrad}) compact quotients. However, it is conjectured \cite{zeghib_AntideSitter} that for $n>3$, all compact quotients are standard. For recent developments on proper actions in the constant curvature case, see \cite{Kassel-Gueritaud-2015, Kassel-Danciger-2016}. 
\end{remark}

 Our aim in the present article is to generalize Theorem \ref{Introduction-Theorem: flat case} and Theorem \ref{Introduction-Theorem: parallel fields} to the case where Minkowski space is replaced by a plane wave spacetime.  Plane waves can be thought of as a generalization as well as a deformation of Minkowski spacetime. They are of great mathematical and physical interests, which can be seen from the amount of recent research on the topic. However, results on compact plane waves are rare since most of research is from the physical point of view, and compact Lorentzian manifolds have a bad causal behavior. More precisely, as stated in the beginning of the section, we consider here compact locally homogeneous  plane waves. Since compact plane waves are complete by \cite[Theorem B]{Leis-Sch} (see also \cite[Theorem 1.2]{mehidi2022completeness} for completeness in the more general context of compact Brinkmann spacetimes), the universal cover of a compact locally homogeneous plane wave is homogeneous.

\subsection{Homogeneous plane waves}\label{Subsection: Grho geometry} A Lorentzian manifold with a lightlike parallel vector field $V$ is called a Brinkmann manifold. The orthogonal distribution $V^\perp$ is integrable and defines a foliation denoted by $\mathcal{F}$, having lightlike geodesic leaves. Plane waves are particular Brinkmann spaces, defined as follows:

\begin{definition}\label{defpw}
A plane wave is a Brinkmann manifold such that the leaves of $\F$ are flat, and $\nabla_X R = 0$, for any $X$ tangent to $V^{\perp}$, where $R$ is the Riemannian tensor. 
\end{definition}

The $(2n+1)$-dimensional Heisenberg group $\Heis_{2n+1}=\R^n \ltimes \R^{n+1}$ is the subgroup of $\Aff(\R^{n+1})$ defined by
$$\Heis_{2n+1}=\left\{\begin{pmatrix}
    1 &\alpha^{\top} \\
    0 &I_n
\end{pmatrix} | \  \alpha \in \R^{n} \right\}\ltimes \R^{n+1}.$$
Denote by $A^+=\R^n$ the abelian subgroup of unipotent matrices, and by  $A^-$ the subgroup $\{0\} \times \R^n$ of the translation part. 

\begin{definition}[Affine unimodular lightlike group]\label{Definition-Introduction: Affine unimodular lightlike group}

Let $\R^{n+1}$ with coordinates $(x_0,x_1,..,x_n)$ be endowed with the  lightlike quadratic form  $q_0 := x_1^2+..+ x_n^2$.

The group of affine isometries of $q_0$ preserving a lightlike vector is $\mathsf{L}_{\u}(n):= (\O(n) \ltimes \R^n) \ltimes \R^{n+1}$. 
$$\mathsf{L}_{\u}(n):=\left\{\begin{pmatrix}
1 & \alpha^{\top}\\
0 & A 
\end{pmatrix} | \  \alpha\in \R^{n}, A\in \O(n) \right\}\ltimes \R^{n+1},$$
 and $\R^n \ltimes \R^{n+1} \subset \mathsf{L}_{\u}(n)$ is the Heisenberg group.  
It  will be called the \textit{affine unimodular lightlike group}. It can also be seen as the group of diffeomorphisms of $R^{n+1}$ preserving $q_0$, a flat affine connection, and a lightlike vector.
\end{definition}
A manifold modeled on $(\mathsf{L}_{\u}(n),\R^{n+1})$ will be said to have an affine unimodular lightlike geometry  in the sense of geometric structures $(\text{see \cite[Chapter 3]{thurston2022geometry} and \cite{goldman2022geometric}})$. \\

Plane waves admit an isometric infinitesimal action of the Heisenberg algebra (see \cite[Section 3.2]{Blau}), whose action preserves individually the leaves of $\mathcal{F}$ and is locally transitive on each $\mathcal{F}$-leaf. Thus, general plane waves have already local cohomogeneity $1$. Moreover, since the leaves of $\F$ are flat and lightlike, with a tangent parallel lightlike vector field $V$, they have an affine unimodular lightlike  geometry.  

Let $X$ be a non-flat simply connected homogeneous plane wave of dimension $n+2$. The connected component of the  isometry group of $X$ is computed in \cite[Theorem 5.13]{Content1} and has the following form  
\begin{equation*}\label{Eq: G_rho}
    G_{\rho}= (\R \times K) \ltimes_{\rho} \Heis_{2n+1},
\end{equation*}
where $K$ is a closed subgroup of $\SO(n)$ and $\rho$ is a suitable homomorphism from $\R\times K$ to $\Aut(\Heis_{2n+1})$. The space $X$ identifies with the quotient  $X_{\rho} =G_{\rho}/I$, with $I = K \ltimes A^+$.  And the codimension $1$ foliation $\mathcal{F}$ is given by the left action of the normal subgroup $K \ltimes \Heis_{2n+1}$, it is invariant by the left action of $G_\rho$. More details are given in Section \ref{Section: Preliminary facts}.

There are many works on  general homogeneous plane waves, for instance \cite{Blau}, \cite{Leis}, but none of these is interested in compact quotients of such spaces. A systematic study of compact quotients of Cahen-Wallach spaces is carried out by Kath and Olbrich in \cite{KO}. These spaces, first introduced in  \cite{cahen1970lorentzian}, are exactly the indecomposable symmetric plane waves.  In this case, the $\rho$-action in the semi-direct product $G_{\rho}$ is semi-simple. For general locally homogeneous plane waves, this is not necessarily the case.
Compact quotients of general homogeneous plane waves are considered in the recent paper \cite{allout2022homogeneous}, in dimension~$3$.

\subsection{Fundamental groups of compact quotients}

In the affine case, i.e. $M=\Gamma\backslash X$, where $G=\Aff(\R^n)$ and $X=\Aff(\R^n)/\GL_{n}(\R)$, a conjecture of Auslander states that the fundamental group of a compact complete affine flat manifold is virtually solvable. It is shown to be true in dimension $3$ by Goldman and Fried~\cite{fried1983three}. 

In the flat Riemannian setting and flat Lorentzian setting, which are particular affine geometries, results about the fundamental group of compact quotients have been already presented in Subsection~\ref{Subsection: Flat case}.

Now let us turn from the flat to the curved case, in particular to the Lorentzian one. We start by considering compact quotients of 
Cahen-Wallach spaces: a classification of fundamental groups of those quotients, up to finite index, is achieved in  \cite[Proposition 8.3]{KO}.

The latter family of symmetric spaces is contained in the class of locally homogeneous plane waves. Recall the definition of $G_\rho$ and $X_\rho$ from the previous subsection. Any compact locally homogeneous plane wave is complete and therefore (up to a finite cover) a Clifford-Klein form of a particular $(G_\rho, X_\rho)$-geometry. For a compact quotient of a general $X_\rho$ by a discrete subgroup $\Gamma \subset G_\rho$ we show the following:

\begin{theorem}\label{Introduction-Theorem: fundamental group}
The fundamental group $\Gamma$ of a compact quotient $\Gamma\backslash X_{\rho}$ is virtually nilpotent, or virtually a nilpotent extension of $\Z$ by a discrete subgroup of Heisenberg, i.e. virtually $\Z\ltimes \Gamma_0$ where $\Gamma_0\subset \Heis_{2n+1}$.
\end{theorem}

\subsection{\textbf{Standardness and semi-standardness of compact quotients}}

When $G$ preserves a Riemannian metric on $X=G/I$ (which is equivalent to the isotropy $I$ being compact), the discrete subgroups of $G$ acting properly and cocompactly on $X$ are exactly the uniform lattices of $G$. For general homogeneous spaces, the isotropy is not compact. Then the first source of examples of compact quotients are the standard quotients, see Definition \ref{Definition: standard}. In this case, $\Gamma$ is a uniform lattice in some connected Lie subgroup $N$ of $G$ acting properly and cocompactly on $X$. In particular, $N$ (which is necessarily closed in $G$) is a syndetic hull of $\Gamma$, see Definition \ref{Definition: syndetic hull}. For some $X=G/I$, all compact quotients are standard and can even be obtained by using the same $N$ (up to conjugacy) for all discrete subgroups $\Gamma\subset G$ that act properly and cocompactly on $G/I$ (up to taking a finite index group). This happens for instance in the flat Riemannian case, where by Bieberbach's theorem $N$ is the group of all translations. In general, not all quotients are standard, and for the standard ones, $N$ depends on $\Gamma$.
 Finding such an $N$ turns out to be an easier problem, and the existence of compact quotients reduces to an existence theorem for lattices.

As stated in Theorem \ref{Introduction-Theorem: flat case}, compact quotients in the flat Minkowski space are virtually standard \cite{goldman1984fundamental},  i.e., they become standard if we replace the discrete group $\Gamma$ by a suitable finite index group.  More generally, a theorem of Fried and Goldman \cite[Section 1.4]{fried1983three} states that any virtually solvable subgroup of the affine group, acting properly discontinuously on the affine space, has a syndetic hull (up to finite index). In simply connected nilpotent groups, the existence of a syndetic hull is due to Malcev \cite{mal1949class}, and is called the Malcev closure. Unlike the previous cases, homogeneous plane waves are not affine manifolds. However, they are foliated by codimension one affine leaves, which are the orbits of the affine unimodular lightlike group. 
The group $G_{\rho}$ may sometimes be solvable, but generically, it is a Lie group which is not even solvable. However, even in the solvable case, there is no construction analogous to the Malcev closure, and in this case we have non-standard examples \cite{maeta2022} (see also Section \ref{Section: Non Standard}). 
We prove the following theorem.
\begin{theorem}\label{Introduction-Theorem: standarness and semi-standardness}
	Any compact quotient of $X_{\rho}$ is  standard or semi-standard. In the standard case, there is a syndetic hull $N$ which is nilpotent or an extension of $\R$ by a subgroup of $\Heis$. Moreover, $N$ acts transitively.
\end{theorem}
For the definition of semi-standard, see Definition \ref{Definition: semi-standard}. Note that in dimension three, all compact quotients are standard \cite[Theorem 12.4]{allout2022homogeneous}.
\subsection{More general homogeneous structures}
In the study of the fundamental group and standardness question, we did not make any use of the geometry of $X_{\rho}$. In particular, $G_{\rho}$ does not have to preserve any Lorentzian metric on $X_{\rho}$.
So Theorem \ref{Introduction-Theorem: fundamental group} extends to more general homogeneous structures (see Theorem \ref{Theorem: General homogeneous}). The property of being standard or semi-standard in Theorem \ref{Introduction-Theorem: standarness and semi-standardness} also remains true, the only thing we lose is the transitivity  of the action of the syndetic hull. 

Based on this remark, even if the initial motivation was to study the compact quotients of homogeneous plane waves, one can consider more general groups
$$G = (\R \times K) \ltimes \Heis, \ \  X = G/I,$$
with no restriction on the action, and the isotropy given by $I = C \ltimes A^+$, where $C$ is a subgroup of $K$ preserving $A^+$. A natural question in future work would be to see which of these homogeneous spaces are Lorentzian, or, more generally, which geometries on $X$ are preserved by $G$. 

\subsection{Equicontinuity of the parallel flow}
The isometry groups considered here are noncompact Lie groups.  Let $(M,g)$ be a compact Lorentzian manifold. A $1$-parameter group of $\Isom(M,g)$ is equicontinuous if its closure in $\Isom(M,g)$ is compact.  This is equivalent to the fact that it is isometric for  some Riemannian metric on $M$.
In a previous work \cite{mehidi2022completeness}, the third and fourth authors asked the question of the equicontinuity for the parallel flow of a compact Brinkmann manifold:
\begin{question}
Let $(M,g,V)$ be a compact Brinkmann spacetime. Is the flow $\phi^t$ of $V$ equicontinuous? 
\end{question}
The condition that  $\Isom(M,g)$ contains a one-parameter group which is not relatively compact in $\Isom(M,g)$ amounts to the non-compactness of the connected component $\Isom^{\circ}(M,g)$.  	

 The results of \cite{mehidi2022completeness} show that the flow $\phi^t$ of an arbitrary compact Brinkmann space is equicontinuous if it is equicontinuous for any compact locally homogeneous one.
We think it would be interesting to ask this question first in the case of plane waves. This would be an important step towards the general Brinkmann case. As a corollary of Theorem \ref{Introduction-Theorem: standarness and semi-standardness}, we obtain that the following:
	\begin{theorem}\label{Introduction-Theorem: equicontinuity of parallel flow}
	Let $(M,V)$ be a compact locally homogeneous plane wave. The action of the parallel flow $V$ is equicontinuous.
\end{theorem}
\textcolor{black}{In the locally symmetric indecomposable case (Cahen-Wallach spaces), the flow is proved to be periodic \cite[Proposition 8.2]{KO}}. In the homogeneous general plane wave case, there are non-periodic examples (Appendix \ref{Appendix B}).

\subsection*{Organization of the article} The article is organized as follows: in Section \ref{Section: Flat case}, we prove existence of parallel vector fields for any flat compact Lorentzian manifold. In Section \ref{Section: Preliminary facts} we give a description of the isometry group of a non-flat simply connected homogeneous plane wave. Section \ref{Section: Fundamental group} deals with the fundamental group of compact quotients of such manifolds, where we prove Theorem \ref{Introduction-Theorem: fundamental group}. In Section \ref{Section: syndetic hull, standard semi-Standrad}, we prove Theorem \ref{Introduction-Theorem: standarness and semi-standardness} about standardness and semi-standardness of compact quotients, and we consider more general homogeneous structures in Section \ref{Section: More homogeneous structures}. This allows to prove equicontinuity of the parallel flow in Section \ref{Section: Equicontinuity} (see also Appendix \ref{Appendix B} for an example of a non-periodic action). Appendix \ref{Appendix A} is related with Section \ref{Section: syndetic hull, standard semi-Standrad}: we show that a cocompact proper action on a contractible space of any connected Lie group admitting a torsion free uniform lattice is transitive. In Section \ref{Section: Non Standard}, we study the non-standard phenomena, starting with a concrete example of a non-standard compact quotient.
\subsection*{Acknowledgment}
We thank the referee for the valuable comments and suggestions that helped improve the quality of the presentation.

\section{Flat case}\label{Section: Flat case}

This section is devoted to  the proof of the first part of Theorem  \ref{Introduction-Theorem: parallel fields}, namely, the existence of a parallel vector field and parallel lightlike direction on any compact flat Lorentzian manifold. For the proof of the equicontinuity statement, see Section~\ref{Section: Equicontinuity}.

Let $M$ be a compact flat Lorentzian manifold of dimension $n+1$, and $\Gamma$ its fundamental group. By item (3) of Theorem \ref{Introduction-Theorem: flat case} we have $M=\Gamma \backslash L$  up to a finite covering, where $L$ is a solvable connected subgroup of the  Poincar\'e group $\O(1, n) \ltimes \R^{n+1}$ that acts simply transitively on $\R^{n+1}$ and $\Gamma$ is a lattice in $L$. Let $\pi: \O(1, n) \ltimes \R^{1+n} \to \O(1, n)$ be the linear part projection, and let $L^\prime$ be the projection of $L$.
Recall that by a direction we mean an oriented line, i.e. an oriented one-dimensional linear subspace. Observe that $M$ has a parallel vector field (resp. a parallel  field of lightlike directions) if and only if $L^\prime$ preserves some vector $v\not=0$ (resp. a lightlike direction $l$ in $\R^{1+n}$). In \cite{grunewald1988transitive},  Grunewald and Margulis give a precise description of all simply transitive groups of affine Lorentz motions on $\R^{1+n}$ using theorems by Auslander. Propositions 5.1 and 5.3 in their paper ensure the existence of $v$ and $l$ for all such groups. In our situation, $L$ is unimodular. We will use this additional information to give a simpler and more direct proof of the existence of $v$ and $l$.

\bigskip

 Before we start, let us introduce some notation.
Let $\mathsf{Pol}$ be the subgroup of  elements of $ \O(1, n) \ltimes \R^{1+n}$  whose linear part preserves a fixed lightlike direction. It has the form $\mathsf{Pol} =  L(\mathsf{Pol}) \ltimes \R^{1+n}$, where the linear part $L(\mathsf{Pol})$ is the group $\Sim_{n-1}$ of similarities 
of $\R^{n-1}$. This is a semi-direct product $L(\mathsf{Pol}) =(\R \times \O(n-1)) \ltimes \R^{n-1}$, whose radical $ \R \ltimes \R^{n-1}$ is the group of affine homotheties of $\R^{n-1}$.  The linear part of $\mathsf{Pol}$ is equal to
$$L(\mathsf{Pol})=\left\{\begin{pmatrix}
    e^{\alpha} & \beta^{\top} & -\frac{\vert \beta \vert^2}{2}\\
0 & A &  - A\beta\\
0 & 0 & e^{-\alpha}
\end{pmatrix} |\;\; \alpha \in \R, A \in \O(n-1), \beta \in \R^{n-1}\right\}.$$
\begin{definition}\label{Definition: Pol and SPol}
 $\mathsf{(1)}$   The group  $\mathsf{Pol}:=((\R \times \O(n-1)) \ltimes \R^{n-1}) \ltimes \R^{1+n}$ above will be referred to as the  \textit{polarized Poincar\'e} group. 
   
  $ \mathsf{(2)}$   The subgroup  $\mathsf{SPol} :=(\O(n-1) \ltimes \R^{n-1})\ltimes \R^{n+1}$ of  elements of $\mathsf{Pol}$  whose linear part preserves a lightlike vector will be referred to as the special polarized Poincar\'e group.  
\end{definition}

\noindent We introduce the following notations, which we will keep throughout this section: 
\begin{itemize}
  
    \item[$\bullet$]  $H:=\{{\rm diag}(e^t, I_{n-1}, e^{-t})\mid t\in\R\}$

    \item[$\bullet$] $K:= \O(n-1)$
    \item[$\bullet$] $U:= \R^{n-1} \subset L(\mathsf{Pol})$ the unipotent radical of $L(\mathsf{Pol})$. 
    \item[$\bullet$] $T:=\R^{n+1}$ the translation part of $\mathsf{Pol}$.
\end{itemize}
 
\noindent Then $\mathsf{Pol}= ((H \times K) \ltimes U) \ltimes T$. 

 In the proof of Theorem \ref{Theorem: solvable implies invariant vector} we will need the following fact on a $1$-parameter subgroup $g^t$ of $\mathsf{Pol}$ which has non-trivial projection to $H$. Let $g^t= h^t  k^t u^t$ be its Jordan decomposition, where $h^t$ the hyperbolic part, $k^t$ the elliptic part, $u^t$ the unipotent part, are $1$-parameter groups in $\mathsf{Pol}$. Up to conjugacy in $\mathsf{Pol}$, we can assume that $h^t$  is in $H$, and since $k^t$ and $h^t$ commute, $k^t$ is then a $1$-parameter group in $K$. Define $\mathfrak{p}:= \Span(W, Z)$ the timelike $2$-plane in $T=\R^{n+1}$ where $h^t$ acts by a hyperbolic matrix. Since $U \ltimes \R^{n+1}$ is normal, the unipotent part $u^t$ is always there, and since $u^t$ and $h^t$ commute, it is  contained in $\mathfrak{p}^\perp \subset \R^{n+1}$. 
\begin{theorem}\label{Theorem section flat: parallel fields}
Up to a finite cover, any compact flat Lorentzian manifold  $M$ admits a parallel vector field, and a parallel field of lightlike directions. 
\end{theorem}

The first part of Theorem \ref{Theorem section flat: parallel fields} follows from Theorem \ref{Theorem: solvable implies invariant vector} below (which is stronger since we don't assume that $L$ has a lattice). 
\begin{theorem}\label{Theorem: solvable implies invariant vector}
Let $L$ be a unimodular solvable Lie  subgroup of $\O(1, n) \ltimes \R^{n+1}$ acting simply transitively on $\R^{n+1}$. Then $L$ preserves some vector field on $\R^{n+1}$.
\end{theorem}
The following lemma will be used in the proof of Theorem \ref{Theorem: solvable implies invariant vector}. 

\begin{lemma}\label{Lemma: solvable in O(n)xR^n}\cite[Propositions 4.2, 4.3] {grunewald1988transitive}
Let $S$ be a subgroup of $\O(n) \ltimes \R^n$. 
Then $S$ acts simply transitively on $\R^n$ if and only if it is generated by pure translations $E:=S \cap \R^n$ and a graph $\phi: E^\perp \to \O(E) \times E^\perp, \phi(v)=(k_v,v)$. In particular, the linear part of $S$ fixes some vector in $\R^n$. 
\end{lemma}

\begin{proof}[Proof of Theorem \ref{Theorem: solvable implies invariant vector}]

    \textsc{Step 1: existence of an invariant line:}  We know that $L$ is solvable, 
   hence $L^\prime \subset \O(1, n)$ is also solvable.  By Lie's theorem, $L^\prime$ preserves a line or a 2-plane. In the first case, $L'$ even preserves a direction $l$ since $L'$ is connected.
   If $L^\prime$ preserves a 2-plane $p$, and $p$ is Lorentzian or 
   degenerate, then it has two or one lightlike direction, which will be $L^\prime$-invariant. If $p$ is spacelike, then we consider the $L^\prime$-action on $p^\perp$ which is Lorentzian. We repeat the process, and surely arrive to either some $2$-plane which is not spacelike, or to a line.
   In all cases, $L^\prime$ preserves a direction $l$. \\\\
\textsc{Step 2: existence of an invariant vector:} If $l$ is not lightlike, then $L^\prime$ preserves a unit vector on it. It  then remains to consider the case where  $l$ is lightlike.  
If $ L$ is contained in $\mathsf{SPol}$, then {\blue $L'$} obviously preserves a lightlike vector.
Henceforth, we assume that $ L$ is not contained in $\mathsf{SPol}$. Then (up to conjugacy in $\mathsf{Pol}$) $L$ contains a $1$-parameter group $g^t=h^t k^t u^t$, with $h^t $ in $H$. 
The adjoint action of $h^t$ on the Lie algebra of $\mathsf{Pol}$ is as follows: 
\begin{itemize}
    \item[$\bullet$]  $\Ad_{h^t}(W)=e^{-t} W, \;\;\Ad_{h^t}(Z) = e^t Z$, where $\mathfrak{p}:= \Span(W, Z)$.  
    \item[$\bullet$] $\Ad_{h^t}(X)=e^t X$, for any $X \in \mathfrak{u}$.
    \item[$\bullet$] $\Ad_{h^t}(Y)=Y$, for any $Y \in \mathfrak{p}^{\perp}$.
    \item[$\bullet$] $\Ad_{h^t}(B)=B$ for any $B \in \mathfrak{k}$.
\end{itemize}
We have that $\Ad_{g^t}$ preserves $\mathfrak{l}:=\mathrm{Lie}(L)$.  Since $h^tk^t$ is the semisimple part of the Jordan decomposition of $g^t$, $\Ad _{h^tk^t}$ also preserves $\mathfrak{l}$, which implies $\Ad_{h^t}(\mathfrak{l})=\mathfrak{l}$. Moreover, since $L$ is unimodular, the adjoint action  of $h^t$ restricted to $\mathfrak{l}$ is unimodular. This gives restrictions on the possible eigenvectors of $\Ad_{h^t}$ inside $\mathfrak{l}$, and hence allows to have some precise information on $L$, which we unroll as follows:   

1. $\mathfrak{p}=\Span(W,Z)$ is contained in $\mathfrak{l}$.

\noindent  Indeed, $L$ acts transitively on $\R^{n+1}$, in particular $L(0)=\R^{n+1}$. Hence there exists an element $V\in(\mathfrak{h}\oplus\mathfrak{k})\ltimes\mathfrak{u}$ such that $V+W$ is in $\mathfrak{l}$. On the other hand, $\mathfrak{l}$ decomposes into eigenspaces of $\Ad_{h^t}$. So we can write $V+W$ as a linear combination of eigenvectors. Since $V$ belongs to $(\mathfrak{h}\oplus\mathfrak{k})\ltimes\mathfrak{u}$, where $\Ad_{h^t}$ has only eigenvalues $1$ and $e^t$, $W$ itself must belong to $\mathfrak{l}$.  Then $e^t$ is also an eigenvalue of $\Ad_{h^t}$ of multiplicity $1$, hence $\lambda Z+X \in \mathfrak{l}$ for some $\lambda\in\R$ and $X \in \mathfrak{u}$. But $X$ is necessarily zero. Otherwise, the action of the $1$-parameter group in $U \ltimes \R^{n+1}$ generated by $\lambda Z+X$ on  $\R^{n+1}$  would have a fixed point, which contradicts the simple transitivity assumption on the $L$-action. 

2. $L$ is contained in $(H \times K) \ltimes T$, i.e. has no element with non-trivial projection to $U$. 

The only other eigenvalue of $\Ad_{h^t}$ is $1$, and the eigenspace $\mathfrak{q}$ is contained in $(\mathfrak{h} \oplus \mathfrak{k}) \ltimes \p^{\perp}$ ($\mathfrak{p}$ is invariant by the adjoint action of $H\times K$, hence also $\mathfrak{p}^\perp$). This, together with the first point, implies that  $\mathfrak{l}$ is contained in $(\mathfrak{h} \oplus \mathfrak{k}) \ltimes \R^{n+1}$, hence claim 2.

3. $L'$ preserves a spacelike vector in $\mathfrak{p}^\perp$. 

We consider the group $Q:=\{q\in L\mid h^t qh^{-t}=q\}$, whose Lie algebra is $\mathfrak q\subset (\mathfrak{h} \oplus \mathfrak{k}) \ltimes \p^{\perp}$. Let $P\subset L$ be the connected subgroup with Lie algebra $\mathfrak p$ and $P^\perp\cong \R^{n-1}$ the one with Lie algebra $\mathfrak p^\perp$. Then $L=Q\ltimes P$. Since $L$ acts simply transitively on $\R^{n+1}$, $Q$ must act simply transitively on $\mathfrak p^\perp$. To verify this it is sufficient to show that $Q$ acts transitively on $\mathfrak p^\perp$. Take $x\in\mathfrak p^\perp$. Then there exists an element $pq\in L$, $p\in P$, $q\in Q$ such that $pq(0)=x$. Since $q(0)$ and $x$ are in $\mathfrak p^\perp$, $p$ is the identity map. Hence $Q$ acts transitively on $\mathfrak p^\perp$.  
Since the action of $H$ on $\mathfrak p^\perp$ is trivial, the projection of $Q$ to $\O(n-1) \ltimes P^\perp$  also acts simply transitively on $\mathfrak p^\perp \cong \R^{n-1}$, which implies, using Lemma \ref{Lemma: solvable in O(n)xR^n}, that the linear part of $Q$ fixes some vector in $\mathfrak p^\perp$ (necessarily spacelike). However, the linear part of $L$ coincides with the linear part of $Q$, i.e.  there is a spacelike vector in $\mathfrak p^\perp$ fixed by $L'$. 
\end{proof}

\begin{proof}[Proof of Theorem $\ref{Theorem section flat: parallel fields}$]
 Let $M^{n+1}$ be a compact, connected flat Lorentzian manifold. By Theorem 1.3, there is a finite index subgroup $\Gamma'$ of $\Gamma:=\pi_1(M)$ that has a syndetic hull $L\subset \O(1,n)\ltimes \R^{n+1}$ which acts simply transitively on the universal cover of $M^{n+1}$, i.e. a finite cover of $M$ is isometric to $\Gamma' \backslash L$. The existence of a parallel vector field on the finite cover $\Gamma' \backslash L$ follows from Theorem \ref{Theorem: solvable implies invariant vector}.

 To prove the existence of a parallel field of lightlike directions, observe first that when there is a timelike vector which is $L'$-invariant, the claim is a direct consequence of Bieberbach's theorem. Indeed, up to finite index, $\Gamma$ is contained in the translation part $\R^{n+1}$ of the isometry group, hence any constant vector field of $\Mink^{n,1}$ induces a parallel vector field on the quotient. To conclude, we proceed as in Step 1 of the previous proof. If the linear part $L'$ preserves a lightlike or timelike line or plane, then it preserves a lightlike direction, up to taking a finite cover. Otherwise, it preserves a maximal spacelike subspace, but repeating the process again on its orthogonal gives an invariant line or plane, which is lightlike or timelike. 
\end{proof}
\begin{remark}
We proved that any compact flat Lorentzian manifold $M= \Gamma \backslash \Mink$ admits a parallel vector field. When the latter is lightlike, $M$ is a flat plane wave and $\Gamma$ preserves a lightlike vector field. So compact flat  plane waves are the compact quotients of $\Mink$ with fundamental group contained in $\mathsf{SPol}$. 
\end{remark}

In the following we give examples of tori with a non-equicontinuous spacelike parallel flow (see Theorem~\ref{Introduction-Theorem: parallel fields}), which is  in fact Anosov in dimension $3$ and partially hyperbolic in higher dimension. These examples admit a lightlike parallel line field, but no lightlike (neither timelike) parallel vector field.
\begin{example}\label{Example: SOL-Anosov} 
Let $q = \Sigma a_{ij} x^i x^j$ be a Lorentzian quadratic form on $\R^{n}$. Consider the flat 
      Lorentzian torus $(\T^n= \R^n/ \Z^n, q)$. Its isometry group is generated by translations together 
      with linear transformations $\O_\Z(q)= \O(q) \cap \GL(n, \Z)$, where $\O(q)$ is the orthogonal group of $q$.
      Let $h \in \O_\Z (q)$ be a partially hyperbolic matrix, i.e. there exists a $2$-dimensional $h$-invariant space of signature $(1,1)$ where the restriction of $h$ has eigenvalues different from $\pm1$. Let $h^t$ be a one parameter group $ \subset \O(q)$ such that $h^1 = h$. 
      Consider the suspension $M$  of $h$, that is $\T^n \times [0, 1]$, where $(x, 1) $ is identified with $(h(x), 0)$. 
      Endow it with the product Lorentzian (flat) metric $q + dt^2$. Then $\frac{\partial}{\partial t}$
      is a spacelike parallel vector field. When $n=2$, the $\frac{\partial}{\partial t}$-flow is Anosov. For $n\geq 3$, the $\frac{\partial}{\partial t}$-flow is partially hyperbolic. In both cases, the $\frac{\partial}{\partial t}$-flow is non-equicontinuous. 
 \end{example} 
 
\section{Preliminary facts on the isometry group}\label{Section: Preliminary facts}

Plane waves (and Brinkmann spacetimes in general) admit locally what is called Brinkmann coordinates, in which the metric has a particular form. When such coordinates exist globally on $J \times \R^{n+1}$, for an open interval $J$, we refer to it as a ``plane wave in standard form".
It is known that the Lie algebra of Killing fields of an indecomposable plane wave in standard form contains the Heisenberg algebra $\heis_{2n+1}$, which acts locally transitively on $\{u\}\times \R^{n+1}$ for all $u\in J$. 
In the homogeneous case, 
Blau and O'Loughlin \cite{Blau} determined the Lie algebra of Killing fields of a plane wave in standard form by analysis of the Killing equation, and classified plane waves in standard form that are homogeneous.
Note, however, that they only consider infinitesimal isometries. 
In \cite[Proposition 5.1]{Content1}, we show that for (general) simply connected non-flat homogeneous plane waves, the infinitesimal action of the Heisenberg algebra integrates to an isometric action of the Heisenberg group. Moreover, we compute the identity component of their isometry group. It turns out \cite{Content1} that these spaces coincide with those found in \cite{Blau}, i.e. these spaces admit global Brinkmann coordinates (see \cite[Section 6]{Content1}).

Let $(X,V)$ be a non-flat simply connected homogeneous plane wave of dimension $n+2$.  
In \cite[Theorem 1.5]{Content1} it is shown that the identity component of the isometry group of a simply connected non-flat homogeneous plane wave of dimension $n+2$ has the form
$$G_\rho:=(\R \times K) \ltimes_{\rho} \Heis_{2n+1}. $$
This result holds independently of the indecomposability of the plane wave. Here $K$ is a closed subgroup of $\SO(n)$ and $\rho$ is a morphism $\rho: \R\times K \to \Aut(\Heis)$, where $\rho$ restricted to $\R$ is given by $\rho(t)= e^{t{\bf L}}$,  ${\bf L} \in \Der(\heis)$, and $\rho(k)$ is the identity on the center of $\Heis$ and equals the standard action of $k$ on $A^+$ and $A^-$ for all $k\in K$. \\
Then $X$ identifies with a quotient $X_{\rho}:= G_{\rho} /I$, with the isotropy given by $I:= K \ltimes A^+$. Conversely, the $\rho$-actions for which $G_{\rho}$ preserves a Lorentzian metric on $G_{\rho}/I$ are characterized in \cite[Proposition 5.6, 5.8]{Content1}, and in this case, the Lorentzian space is necessarily a plane wave.
Accordingly, throughout this article, we will refer to a homogeneous  plane wave as a homogeneous space of the form $X_{\rho}$. 

In \cite{Blau} non-flat homogeneous plane waves come in two families. Using the notation introduced above, the two families differ in that $L$ acts trivially on the center of $\heis$ in one case and non-trivially in the other.
In other words, the center of $G_\rho$ is non-trivial in the first case and trivial in the second. The homogeneous plane waves of the first family are complete, those of the second family are incomplete.

\begin{fact}\label{Fact: parallel flow is Z}
    The parallel vector field $V$ is a generator of the center of $\heis$.
\end{fact}   
\begin{proof}
Let $z$ be a generator of the center of $\heis$. The Lorentzian scalar product on $T_{o}(G_{\rho}/I)$ is $\Ad_h$-invariant for any $h \in A^+$. Consider a non-trivial $h \in A^+$. The action of $\Ad_h$ on $\mathfrak{g}_\rho / \mathfrak{i}$ is a unipotent matrix for which $z$ is the eigenvector of eigenvalue $1$. This implies that $z$ is necessarily lightlike. Moreover, $\Heis$ contains an abelian subgroup (whose Lie algebra contains $z$) acting transitively on the leaves of $V^\perp$. So by \cite[Theorem 3]{Leis}, the vector field on $G_\rho/I$ induced by the action of $z$ is parallel. Now, when $X$ is non-flat, it has a unique parallel lightlike vector field, which is then the one induced by $z$.    
\end{proof}

Let us now make the following observation. Let $G={\Isom}(X_{\rho})$ denote the full isometry group of a non-flat simply connected homogeneous plane wave $X_{\rho}$. The identity component of $G$ is ${\Isom}^{\sf{o}}(X_{\rho})=G_{\rho}$. We have that $G_{\rho}$ has finite index in $G$ (see \cite[Proposition 5.1]{Content1}). So for a compact $(G,X_{\rho})$-manifold, there is a finite cover which is a $(G_{\rho},X_{\rho})$-manifold. Hence, when studying compact quotients, one restricts to $(G_{\rho}, X_{\rho})$-manifolds.

\section{Discrete subgroups}\label{Section: Fundamental group} 
In Sections \ref{Section: Fundamental group}-\ref{Section: Equicontinuity} we consider $G_{\rho}$ to be the connected component of the isometry group of a simply connected non-flat homogeneous plane wave. For notations to be lighter the Heisenberg group $\Heis_{2n+1}$ will be denoted simply by $\Heis$. 
Although we introduced $\Heis$ as a linear group, we will also use the common realization  of $\Heis$ as an extension of $\C^n$ by $\R$ and we will write $\Heis\cong \R \times \C^n$. Under this identification, we have  $A^+\cong \R^n=\{0\}\times \R^n\subset \R \times \C^n$ and $A^-\cong \R^n=\{0\}\times (i\R)^n\subset \R \times \C^n$.

Let $G_{\rho}= (\R \times K) \ltimes_{\rho} \Heis$, $I = K \ltimes A^+$, and $X_{\rho}:= G_{\rho}/I$. Unless otherwise stated, all over the paper, $\R$ acts on $\Heis$ via a morphism $\rho: \R \to \Aut(\Heis)$, with $\rho(t)= e^{t{\bf L}}$,  ${\bf L} \in \Der(\heis)$, and $K$ acts on $\Heis$, trivially on the center, and by standard action on $A^+$ and $A^-$. We can suppose (up to adding $\ad(h)$ for  a suitable $h\in \heis$) that ${\bf L}$ preserves $\mathfrak{a}^+ \oplus \mathfrak{a}^- \subset \heis$, where ${\mathfrak a}^\pm={\rm Lie}(A^\pm)$.\\

In this section, we generalize results from \cite[Section 3]{KO} to the case where $\rho$ is not semi-simple. We will use a different, more conceptual approach, providing the general case directly.\\

\noindent\textbf{Notations.}  Define the projection morphisms $p: G_{\rho} \to \R$, $p_K: G_{\rho} \to K$, and $r: G_{\rho} \to \R \times K$.

Observe that $G_{\rho}$ is a connected amenable Lie group, i.e. a Lie group  having a normal solvable subgroup with compact quotient. Hence any discrete subgroup $\Gamma$ of $G_{\rho}$ is virtually polycyclic (this follows from general properties of connected amenable Lie groups (see \cite[Lemma 2.2]{milnor1977fundamental}). Moreover, it is well known that polycyclic groups are finitely generated, therefore, so is any discrete subgroup of $G_{\rho}$.  The latter property will be used in the proof of the next theorem.\\

\noindent\textbf{Terminology.} Let $\Gamma$ be a discrete subgroup of $G_{\rho}$. We say that $\Gamma$ is \textit{straight} if $p(\Gamma)$ is discrete. Otherwise, it is \textit{non-straight}.\\
\begin{theorem}\label{Theorem Gamma}
       Let $\Gamma$ be a discrete subgroup of $G_{\rho}$. Then 
       \begin{itemize}
           \item[(1)] If $\Gamma$ is non-straight, then it is virtually nilpotent.
           \item[(2)]  If $\Gamma$ is straight, then either $\Gamma\cong \Z \ltimes \Gamma_0$ or $\Gamma = \Gamma_0$, where $\Gamma_0\subset K\ltimes\Heis$ is virtually nilpotent.  If in addition, $\Gamma$ acts properly cocompactly on $X_\rho$, then we are in the case $\Gamma\cong \Z \ltimes \Gamma_0$. Moreover, the subgroup $\Gamma_0$ has a finite projection to $K$. In particular, $\Gamma$ is virtually a nilpotent extension of the integers $\Z$  by a discrete subgroup of $\Heis$.
       \end{itemize} 
\end{theorem}

First we state the following two Lemmas and fact.

\begin{lemma}[Zassenhaus Lemma, Theorem 4.1.6 \cite{Thurston},   
Proposition 8.16 \cite{raghunathan1972discrete}]
    Let $G$ be any Lie group. Then there exists a neighborhood $U$ of the identity such that any discrete subgroup $\Gamma$ generated by $\Gamma\cap U$ is nilpotent. We call such a neighborhood a \textit{Zassenhaus neighborhood}.
\end{lemma}

 \begin{lemma}
     Let $\Lambda$ be a subgroup of a Lie group $G$. Define $\Lambda_0:=\Lambda\cap {\overline{\Lambda}}^{\mathsf{o}}$, where ${\overline{\Lambda}}^{\mathsf{o}}$ denotes the identity component of the (topological) closure of $\Lambda$. Then the subgroup $\Lambda_{0}$ can be generated by $\Lambda_{0}\cap V$ for any $V$ neighborhood of identity in ${\overline{\Lambda}}^{\mathsf{o}}$. 
 \end{lemma}
        \begin{proof}
            Let $V$ be a neighborhood of the identity  in $\overline {\Lambda}^{\mathsf{o}}$. Define $\Lambda_1:= \langle \Lambda_0 \cap V \rangle$.  Then $\Lambda_1$ is dense in $\overline {\Lambda}^{\mathsf{o}}$ since $V\subset\overline{\Lambda_1}\subset\overline{\Lambda_0}=\overline {\Lambda}^{\mathsf{o}}$ and $\overline {\Lambda}^{\mathsf{o}}$ is connected. Now, let $\lambda_0\in \Lambda_0$, by density of $\Lambda_1$ there is $\lambda_1$ such that $\lambda_0\lambda_1^{-1}\in V$. 
        \end{proof}
        The following fact will be used to prove the second part of point (2).
\begin{fact}[Hirsch, Goldman, Fried \cite{fried1981affine} Theorem A]\label{fact2.7}
    Let $M$ be a compact affine manifold with nilpotent holonomy group. Then completeness of $M$ in the sense of $(\Aff(\R^{n}), \R^{n})$-structure is equivalent to the linear holonomy being unipotent.
\end{fact}
     
\begin{proof}[Proof of Theorem $\ref{Theorem Gamma}$]
(1) The restriction of the projection $\R\times K\to \R$ to $\overline {r(\Gamma)}$ is a Lie group homomorphism $\overline {r(\Gamma)}\to \overline {p(\Gamma)}=\R$. It is surjective since $K$ is compact. Thus we obtain a fibration 
$$ \overline{r(\Gamma)}\cap K \hookrightarrow \overline{r(\Gamma)}\to \R.$$
Its fiber has only finitely many components since it is compact. Now we see from the long exact homotopy sequence of the fibration that also $\pi_0(\overline{r(\Gamma)})$ is finite. This implies that $\Lambda_0:={\overline{r(\Gamma)}}^{\mathsf{o}}\cap r(\Gamma)$ has finite index in $r(\Gamma)$. Consequently, also $\Gamma^0:=r^{-1}(\Lambda_0)\cap \Gamma$ has finite index in $\Gamma$. Hence it is sufficient to show that $\Gamma^0$ is nilpotent. Since $\Gamma^0$ has finite index in $\Gamma$, it is also finitely generated. We choose $U_1\subset\R\times K$ and $U_2\subset \Heis$ such that $U_1\times U_2$ is a Zassenhaus neighborhood in $G_\rho$. The above Lemma applied to $\Lambda=r(\Gamma)$ yields that $\Lambda_0$ is generated by $\Lambda_0\cap U_1$. Let $rh$ be one of finitely many generators of $\Gamma^0$, where $r\in\R\times K$, and $h\in\Heis$. Then $r=\lambda_1\cdot\ldots\cdot \lambda_k$ for $\lambda_j\in\Lambda_0\cap U_1$, $j=1,\dots,k$. Choose elements $\gamma_j\in\Gamma^0$ such that $r(\gamma_j)=\lambda_j$, $j=1,\dots,k$. Then 
\[ rh=\lambda_1\cdot\ldots\cdot\lambda_k h=\gamma_1\cdot\ldots\cdot\gamma_k h'\] 
for some $h'\in \Gamma^0\cap\Heis$ since $\Heis$ is normal in $G_\rho$. Thus we may replace $rh$ by $\gamma_1,\dots,\gamma_k$ and $h'$. Doing so for every generator of $\Gamma^0$ we obtain a set of generators $\{r_ih_i\}_{i=1}^m$, where $r_i\in U_1$ and $h_i\in\Heis$.
Let $A$ be a $K$-invariant complement of the center in $\Heis$. We consider the automorphism $\Psi$ of $G_\rho$ which is the identity on $\R\times K$, multiplication by $a\in\R_{>0}$ on $A$ and multiplication by $a^2$ on the center of $\Heis$. Choose $a$ so small that $\Psi(h_i)\in U_2$ for all $i=1,\dots,m$. Then $\Psi(\Gamma^0)=\langle r_i\Psi(h_i),\ i=1,\dots,m\rangle$ is generated by elements of $U_1\times U_2$, thus it is nilpotent. Consequently, also $\Gamma^0$ is nilpotent.\\
(2) If the projection of $\Gamma$ to the $\R$-factor is discrete, it is either trivial or isomorphic to $\Z$, and in both cases, the exact sequence $1 \to \Gamma \cap (K\ltimes\Heis) \to \Gamma \to p(\Gamma) \to 1$ splits. Hence, $\Gamma$ is isomorphic to $ p(\Gamma)\ltimes \Gamma_0$ for $\Gamma_0:= \Gamma \cap (K\ltimes\Heis)$. The group $\Gamma_0$ is virtually nilpotent. To see this, one can proceed as in the proof of (1) replacing $r$ by the projection $p_K$ to $K$: The closure $\overline{p_K(\Gamma_0)}$ has finitely many connected components. Thus $\Lambda_0:={\overline{p_K(\Gamma_0)}}^{\mathsf{o}}\cap p_K(\Gamma_0)$ has finite index in $p_K(\Gamma_0)$ and $\Gamma^0:=p_K^{-1}(\Lambda_0)\cap \Gamma_0$ has finite index in $\Gamma_0$. Now we use a Zassenhaus neighborhood $U_1\times U_2$, $U_1\subset K$, $U_2\subset \Heis$ to see that $\Gamma^0$ is nilpotent. 
To complete the proof of (2), we will show that $\Gamma_0 \cap \Heis$ has finite index in $\Gamma_0$. Since $K\ltimes \Heis$ is a linear group, by Selberg's lemma, $\Gamma_0$ contains a torsion free subgroup $\Gamma_0'$  of finite index. The leaf $\Gamma_0' \backslash K\ltimes \Heis/K\ltimes A^+$ is a compact manifold having a complete affine structure with nilpotent holonomy (Remark \ref{Remark: Heis-leaves have unimodular lightlike structure}).   It follows from Fact \ref{fact2.7} that the linear part of $\Gamma_0'$ is unipotent, meaning that $\Gamma_0'$ has a trivial $K$-part. Thus, $p_K(\Gamma_0)$ is finite. Let $\Gamma_1:= \Gamma_0\cap \Heis$, since $\Heis$ is preserved by the $\Z$-action, we can define $\Z\ltimes \Gamma_1$ which is a nilpotent extension of $\Z$ and clearly of a finite index in $\Gamma$.
\end{proof}
\begin{remark}
The first part of point $(2)$ in Theorem $\ref{Theorem Gamma}$ can also be deduced from Auslander's theorem \cite[Theorem 3]{Auslander1}. Note that the formulation in Auslander's paper is slightly incorrect.
 \end{remark}

\begin{remark}\label{Remark: Heis-leaves have unimodular lightlike structure}
 Let $\Gamma$ be a torsion free discrete subgroup of $G_\rho$ which acts properly and cocompactly on $X_{\rho}$.  Then the quotient space $\Gamma\backslash X_{\rho}$ is a compact manifold foliated by a codimension $1$ foliation, given by the $K \ltimes \Heis$-action. Since the $K \ltimes \Heis$-foliation is defined by a closed $1$-form, the leaves are either all closed or all dense, and they are closed exactly when $p(\Gamma)$ is discrete.
In the straight case, the $K \ltimes \Heis$-leaves are complete compact affine manifolds, modeled on the so-called affine unimodular lightlike geometry (see Definition $\ref{Definition-Introduction: Affine unimodular lightlike group}$). Namely, they have a $(\mathsf{L}_{\u}(n), \R^{n+1})$-structure.
\end{remark}

  \begin{remark}[Weakly plane waves]
Let $(X,V)$ be a simply connected, non-flat homogeneous plane wave. The parallel lightlike vector field $V$ is generated by the center of $\heis$ (see Fact $\ref{Fact: parallel flow is Z}$). Since $G_\rho$ preserves the center of $\Heis$, any quotient $M$ of $X$ by a discrete subgroup of $G_\rho$ inherits a parallel lightlike line field $l$. If $V$ were timelike or spacelike, then, by passing to a time-orientable cover of $M$, one would obtain a global parallel vector field spanning $l$  by taking a constant-length section that is compatible with the orientation. But here $V$ is lightlike, so there is no way to choose a global parallel section a priori. In this situation, time orientation does not help to patch together the existing local parallel sections. 

 Let us consider this in more detail. Recall that there are two families of simply connected non-flat homogeneous plane waves, one in which the center of $\Heis$ is centralized by $G_\rho$, and one in which the $\R$-factor acts non-trivially on the center of $\Heis$. Therefore, if we take a quotient $M$ of a plane wave from the first family, then $V$ descends to a lightlike parallel vector field on $M$ since in this case $V$ is preserved by any discrete subgroup of $G_\rho$. Thus $M$ is also a plane wave.  However, in the second family, if the discrete subgroup has non-trivial projection to the $\R$-factor, then it contains elements sending $V$ to $\lambda V$, with $\lambda \neq \pm 1$. Therefore only the lightlike parallel line field $\R V$ descends to the quotient. Such quotients are weakly plane waves in the sense of Remark~$\ref{remark: Kundt space- weakly plane wave}$.
  The second family is contained in the more general family of pp-waves (not necessarily homogeneous) described in \cite[Section 6.4]{baum2014}. There it is shown that all these pp-waves admit (non-compact) quotients which are time-orientable and admit a parallel lightlike line field but no parallel lightlike vector field.
The question now is whether such a situation can also arise with a compact quotient. If the quotient of a plane wave in the second family is compact, it possesses a lightlike parallel line field but no lightlike parallel vector field. Indeed, 
since the $K \ltimes \Heis$-leaves have codimension $1$ in $X$, the compactness of the quotient implies that the fundamental group has non-trivial projection to the $\R$-factor. Such compact quotients indeed exist (see \cite[Theorem 1.4]{allout2022homogeneous} for examples in dimension $3$). 	\end{remark}

\section{Standardness and semi-standardness}\label{Section: syndetic hull, standard semi-Standrad}

\subsection{Existence of a syndetic hull}
\label{Subsection:Definition of syndetic hull?} 
\begin{definition}[Syndetic hull]\label{Definition: syndetic hull}
    Let $G$ be a Lie group and let $\Gamma$ be a discrete subgroup. A \emph{syndetic hull} of $\Gamma$ in $G$ is a closed connected Lie subgroup $N \subset G$  containing $\Gamma$ such that $\Gamma\backslash N$ is compact. 
\end{definition}
\begin{definition}[$\Z$-Syndetic hull]\label{Definition: Z-syndetic hull}  A \emph{$\Z$-syndetic hull} of $\Gamma$ in $G$ is a closed Lie subgroup $N\subset G$ having an infinite cyclic component group such that $\Gamma\subset N$ and $\Gamma\backslash N$ is compact.
\end{definition}

So far we proved, using the existence of Zassenhaus neighborhoods, that up to finite index, any discrete subgroup of $G_{\rho}$ is either virtually nilpotent, or is an extension of $\Z$ by a virtually nilpotent group. The same proofs, using now what we call here `strong Zassenhaus neighborhood' instead of a Zassenhaus neighborhood, lead to a stronger result, namely the existence of a nilpotent syndetic hull (in the non-straight case), or of a  $\Z$-syndetic hull with a nilpotent identity component  (in the straight case)  for a finite index subgroup.  Here we are especially interested in the situation where the discrete group acts properly and cocompactly on $X_\rho$. We will see that for such a group in the straight case the identity component of the $\Z$-syndetic hull is contained in $\Heis$.

\begin{fact}\cite[Theorem 4.1.7]{Thurston}\label{Thurston: strongly Zassenhaus neighborhood}
Let $G$ be a Lie group. There exists a neighborhood $U$ of $1$ in $G$ such that any discrete subgroup $\Gamma$ of $G$ generated by $U \cap \Gamma$ is a cocompact subgroup of a connected, closed, nilpotent subgroup $N$ of $G$. 	
\end{fact}
We call such an identity neighborhood $U$ a \textit{strong Zassenhaus neighborhood}. 

\begin{proposition}\label{Proposition: syndetic hull for non-straight Gamma}
 Let $\Gamma$ be a discrete subgroup of $G_{\rho}:=(\R \times K) \ltimes_{\rho} \Heis$. If $\Gamma$ is non-straight, then up to finite index $\Gamma$ has a nilpotent syndetic hull $N$. Furthermore, $\Gamma$ acts properly and cocompactly on $X_{\rho}=G_{\rho}/I$ if and only if $N$ acts properly and cocompactly on~$X_{\rho}$.
\end{proposition}
\begin{proposition}\label{Proposition: partial syndetic hull for straight Gamma}
 Let $\Gamma = \langle\hat{\gamma}\rangle \ltimes \Gamma_0$ be a straight discrete subgroup of $G_{\rho}:=(\R \times K) \ltimes_{\rho} \Heis$, acting properly and cocompactly on $X_\rho=G_{\rho}/I$. Then up to finite index $\Gamma_0$ has a nilpotent syndetic hull $N_0$ contained in $\Heis$, which is $\hat{\gamma}$-invariant. 
\end{proposition}

\begin{proof}
The same proof as in Theorem \ref{Theorem Gamma} yields that $\Gamma$ (resp. $\Gamma_0$) is generated by a strongly Zassenhaus neighborhood in the non-straight (resp. straight) case. The claim then follows from  Fact \ref{Thurston: strongly Zassenhaus neighborhood}.  In the straight case, up to finite index, we have $\Gamma = \langle \hat{\gamma} \rangle \ltimes \Gamma_1$, with $\Gamma_1$ a subgroup of $\Heis$ (proof of Theorem \ref{Theorem Gamma} (2)). The Malcev closure of $\Gamma_1$ in $\Heis$ is a syndetic hull for $\Gamma_1$ which is $\hat{\gamma}$-invariant.  
\end{proof}

\subsection{Transitivity of cocompact actions of  Lie groups}\label{Subsection:Transitivity of cocompact actions of  Lie groups}

Let $\Gamma$ be a discrete subgroup of $G_{\rho}$ acting properly on $X_{\rho}\cong \R^{n+2}$. The aim of this section is to prove that the syndetic hull $N$ (resp. $N_0$) of $\Gamma$ (resp.  $\Gamma_0$) obtained in the previous section acts transitively on $X_{\rho}$ (resp. on an $\mathcal{F}$-leaf).

 One could actually ask a more general question: let $G$ be a Lie group acting properly cocompactly on a contractible manifold $X$, does $G$ act transitively? 
 In Proposition \ref{Proposition-Appendix A: transitive action under existence of a lattice}, we prove that the action is transitive if we further assume that $G$ has a torsion free uniform lattice. In this subsection we do not assume that $G$ has a lattice, and prove transitivity when $G$ is a connected nilpotent Lie subgroup of~$G_{\rho}$.

\begin{proposition}\label{Proposition: transitive action along the Heis-leaves, without lattice}
Let $N$ be a connected nilpotent Lie subgroup of $G:= K \ltimes \Heis$, acting cocompactly on the homogeneous space $Y := (K \ltimes \Heis) / (K \ltimes A^+)$. Then 
\begin{itemize}
    \item $N $ acts transitively,
    \item $N$ is contained in $\Heis$, and contains the center of Heisenberg. 
\end{itemize}

\end{proposition}
\begin{proof}
The group $N$ acts also cocompactly on $Y/Z$. Let $p':K\ltimes \Heis\to A^-$  be defined by $z(a^++a^-)k\mapsto a^-$, where $k\in K$, $a^\pm\in A^\pm$ and $z$ is in the center of $\Heis$.  Then $p'$ induces a bijection from $Y/Z$ to $A^-$. Under this identification, $n\in N$ acts on $A^-$ by its projection to ${K\ltimes A^-}$. Hence the projection $\hat N$ of $N$ to $K\ltimes A^- \subset  \O(n)\ltimes \R^n$ acts cocompactly on $A^-\cong \R^n$. The group $\hat K:=p_K (\hat N)\subset \O(n)$ acts trivially on $A_0:=\hat N\cap A^-$ since $\hat N$ is nilpotent. Let $A_1$ be the orthogonal complement of $A_0$ in $A^-$ and define $\phi:\hat K \to A_1$ by $(k,\phi(k))\in\hat N$. Then $K':={\rm graph}(\phi)$ is a subgroup of $\hat N$.  Moreover, $\hat N=K'\times A_0$. Since $K'$ is compact, $A_0$ acts cocompactly on $A^-$. This implies $A_0=A^-$ and $\hat K=\{1\}$. Thus $N$ is contained in $\Heis$ and its projection to $A^-$ is surjective. If $N \cap A^+ \neq \{0\}$, then $N$ contains the center, hence its action on $Y$ is transitive. Otherwise, i.e. if $N \cap A^+ = \{0\}$, then $N$ acts freely properly on $Y$, defining a fibration 
$ N \simeq \R^k \to Y \simeq \R^{n+1} \to N \backslash Y.$
It follows from the long exact sequence of homotopy groups that all the homotopy groups of the compact manifold $N \backslash Y$ are trivial. This implies that it is a point, hence  $N$ acts transitively on $Y$ and has dimension $n+1$. We claim that $N$ contains the center. Indeed, if $N$ is not abelian, then it contains $[N,N] = Z$, the center of $\Heis$. Otherwise, $N$ is an abelian subgroup of dimension $n+1$ of $\Heis$. It necessarily contains the center, since an abelian subgroup of $\Heis$ has maximal dimension $n+1$.  
\end{proof}

\begin{lemma}\label{Lemma: N co-compact implies N_0 co-compact}
Let $N$ be a connected Lie subgroup of $G:=(\R \times K) \ltimes \Heis$ acting cocompactly on the homogeneous space $X := (\R \times K) \ltimes \Heis / K \ltimes A^+$. Then $N_0 := N \cap (K \ltimes \Heis)$ acts cocompactly on  each $(K \ltimes \Heis)$-leaf. 
\end{lemma}
\begin{proof}
Let $W\subset X$ be a compact set such that $N\cdot W=X$. It is sufficient to consider the ($K\ltimes \Heis$)-leaf $\F_0:=p^{-1}(0)$. Since $N$ is connected and acts cocompactly, it contains a one-parameter group $\gamma(s)$ such that $(p\circ\gamma)(s)=s$. Since $p(W)$ is compact, we may assume $p(W)\subset [-a,a]$. Now take $q\in\F_0$. We can choose an element $n\in N$ such that $nq\in W$. Note that $p(n)\in [-a,a]$. Then $n_0:=\gamma(p(n))^{-1} n\in N_0$ and 
$n_0q\in (\gamma(p(n))^{-1}\cdot W)\cap \F_0\subset W_0:=\gamma([-a,a])\cdot W \cap \F_0,$
thus $N_0\cdot W_0= \F_0$. Moreover, $\F_0$ is compact since $[-a,a]$ and $W$ are compact and and multiplication is continuous. 
\end{proof}
\begin{proposition}\label{Proposition: transitive action without lattice}
 Let $N$ be a connected nilpotent Lie subgroup of $G:=(\R \times K) \ltimes \Heis$ acting cocompactly on the homogeneous space $X := (\R \times K) \ltimes \Heis / K \ltimes A^+$. Then $N $ acts transitively. Moreover,  $N_0 := N \cap (K \ltimes \Heis)$ is contained in $\Heis$, and contains the center of Heisenberg. 
\end{proposition}

\begin{proof} This is a straightforward consequence of Proposition \ref{Proposition: transitive action along the Heis-leaves, without lattice} and Lemma \ref{Lemma: N co-compact implies N_0 co-compact}. 
\end{proof}

\subsection{Standardness of compact quotients}

\begin{definition}[Standard quotient]\label{Definition: standard}
	Let $X = G/I$ be a homogeneous space, and $M=\Gamma \backslash	 X$ a compact quotient of $X$ by some discrete subgroup $\Gamma$ of $G$, acting properly  and freely on $X$. We say that the quotient manifold $M=\Gamma \backslash	 X$  (or $\Gamma$) is standard if (up to finite index) $\Gamma$  is contained in some connected Lie subgroup $N$ of $G$ acting properly on $X$.
\end{definition}
If $\Gamma \backslash X$ is standard, then $N$ necessarily acts cocompactly on $X$ and $\Gamma$ is a uniform lattice in $N$. In particular, $N$ is closed in $G$. Hence, the quotient $\Gamma \backslash X$ is standard if and only if $\Gamma$ admits a syndetic hull in $G$. 
\begin{definition}[Semi-standard quotient]\label{Definition: semi-standard}
	Let $X = G/I$ be a homogeneous space, and $M=\Gamma \backslash	 X$ a compact quotient of $X$ by some discrete subgroup $\Gamma$ of $G$, acting properly  and freely on $X$. We say that the quotient manifold $M=\Gamma \backslash X$ (or $\Gamma$) is semi-standard if there exists a normal Lie subgroup $G_0$ of $G$ containing $I$ such that \begin{itemize}
	    \item $\Gamma\cap G_0$ is standard for $G_0/I$.
        \item The projection of $\Gamma$ to $G/G_0$ is a lattice.
	\end{itemize}
\end{definition}
Let $\Gamma$ be a discrete subgroup of $G_{\rho}$ acting properly cocompactly on $X_{\rho}$. The following two theorems are corollaries of Sections \ref{Subsection:Definition of syndetic hull?} and \ref{Subsection:Transitivity of cocompact actions of  Lie groups}.

\begin{theorem}[Non-straight case]\label{Theorem: standarness in non-straight case}
	Any non-straight compact quotient of $X_{\rho}$ is standard. More precisely,  let $\Gamma$ be a non-straight discrete subgroup of $G_{\rho}$ acting properly and cocompactly on $X_{\rho}$. Then, up to finite index, $\Gamma$ has  a nilpotent syndetic hull $N$ in $G_{\rho}$. Moreover :
 \begin{itemize}
     \item $N$ acts simply transitively on $X_{\rho}$,
     \item $N = \R \ltimes N_0$, where $\R$ is a $1$-parameter group of $N$ with non-trivial projection to the $\R$-factor in $G_{\rho}$, and $N_0 := N \cap (K \ltimes \Heis)$. Moreover, $N_0 \subset \Heis$, and  contains the center of Heisenberg. 
 \end{itemize}
\end{theorem}
\begin{theorem}[Straight case]\label{Theorem: standarness in straight case}
	Any straight compact quotient of $X_{\rho}$ is semi-standard. More precisely, 
 let $\Gamma$ be a straight discrete subgroup of $G_{\rho}$ acting properly and cocompactly on $X_{\rho}$ (by Theorem  $\ref{Theorem Gamma} (2)$, up to finite index,  $\Gamma=\Z \ltimes \Gamma_0$, with $\Gamma_0:= \Gamma \cap \Heis$). 
 Then $\Gamma_0$ has a syndetic hull $N_0$ in $\Heis$. Moreover, $N_0$ acts transitively on the $\Heis$-leaves, and contains the center of Heisenberg.
\end{theorem}

We want to study under which conditions a straight compact quotient is even standard. We proceed similarly to \cite[Prop.\,8.18]{KO}. Let us first recall the following notation. The Heisenberg group is written as $\Heis = Z \cdot A$, where $Z=\R$, and  $A:=\C^n=A^+\oplus A^-$. Any connected subgroup $N_0$ of $\Heis$ containing the center can be written as $N_0=Z\cdot A'$ for a subspace $A'\subset A$. Let $(\R\oplus {\mathfrak k})\ltimes \heis$ be the Lie algebra of $G_\rho$. As described in Section \ref{Section: Preliminary facts}, we have $\rho(t)=e^{t{\bf L}}$ for some derivation ${\bf L}$ of $\heis$ that preserves $A$.

\begin{proposition} \label{Proposition : semi/standard}
Let $\Gamma=\langle \hat \gamma \rangle \ltimes \Gamma_0$, $\Gamma_0\subset \Heis$, be a straight discrete subgroup of $G_{\rho}$ acting properly and cocompactly on $X_{\rho}$. Let $N_0=Z\cdot A'\subset \Heis$ be the Malcev closure of $\Gamma_0$ in $\Heis$. Then $\Gamma\backslash X_\rho$ is standard if and only if there are elements $(1,\phi)\in \R\times {\mathfrak k}$, $X\in {\mathfrak n}_0:={\rm Lie}(N_0)$, $\hat t\in\R\setminus\{0\}$, and $\hat n\in N_0$ such that $\hat \gamma = \hat n \exp\big(\hat t\,(1,\phi,X)\big)$ and $A'$ is invariant under ${\bf L}+\phi$.
\end{proposition}
\begin{proof}
The ``if'' part is clear. We show the ``only if'' part. Let $S$ be a syndetic hull of $\Gamma$. The subgroup $r(S)\subset \R \times K$ is connected. Since $K$ is compact, there is a one parameter subgroup $c(t)$, $t\in\R$, of $r(S)$ containing $r(\hat \gamma)$. Let $\phi\in \mathfrak k$ be such that $c(t)=(t,e^{t\phi})$ and $\hat t\in\R$ be such that $c(\hat t)=r(\hat \gamma)$.  Let $\tilde N_0$ be the unique connected subgroup of $\Heis$ for which $(S\cap \Heis)\backslash \tilde N_0$ is compact, see \cite[Prop.\,2.5.]{raghunathan1972discrete}. Because $S\cap \Heis$ acts properly on $\Heis/ A^+$, also $\tilde N_0$ acts properly on $\Heis/ A^+$. 
Since $\Gamma_0\subset S\cap\Heis\subset \tilde N_0$ and $\tilde N_0$ is nilpotent, $\Gamma_0$ admits a unique Malcev hull in $\tilde N_0$, which is also a Malcev hull in $\Heis$. On the other hand, $N_0$ is a Malcev hull of $\Gamma_0$ in $\Heis$, which implies $N_0\subset \tilde N_0$. Moreover, $\tilde N_0$ cannot be larger than $N_0$. Indeed, $A'+A^+ =A$ since $N_0$ acts transitively on the Heisenberg leaves. On the other hand, we have $\tilde N_0=Z\cdot \tilde A$, where $A'\subset \tilde A\subset A$, and $\tilde N_0$ acts properly on $\Heis/ A^+$.  In particular,  $\tilde A\cap A^+=0$, thus $A'=\tilde A$. We obtain $N_0=\tilde N_0$. Since the subgroup $S\cap \Heis$ is normal in $S$, also $N_0=\tilde N_0$ is normalized by $S$. Consequently, $A'$ is invariant under ${\bf L}+\phi$. Since $c(t)$ is contained in $r(S)$, the vector $(1,\phi)$ is in the projection of the Lie algebra of $r(S)\ltimes N_0$ to $\R\times{\mathfrak k}$. Choose $X\in {\mathfrak n}_0$ such that $(1,\phi,X)$ is in the Lie algebra of $r(S)\ltimes N_0$. Then the projection of $\exp\big(\hat t\,(1,\phi,X)\big)$ equals $r(\hat \gamma)$. Since $\hat \gamma$ belongs to $S\subset r(S)\ltimes N_0$, it differs from $\exp\big(\hat t\,(1,\phi,X)\big)$ by an element of $N_0$. 
\end{proof}
As a consequence of Proposition \ref{Proposition : semi/standard}, we obtain that in the standard case, there is a syndetic hull which is an extension of $\R$ by a subgroup of $\Heis$. 
\begin{cor}
    Let $\Gamma=\langle \hat \gamma \rangle \ltimes \Gamma_0$, $\Gamma_0\subset \Heis$, be a straight discrete subgroup of $G_{\rho}$ acting properly and cocompactly on $X_{\rho}$. Let $N_0$ be the Malcev closure of $\Gamma_0$ in $\Heis$. If $\Gamma \backslash X_{\rho}$ is standard, then there is a syndetic hull of $\Gamma$ in $G_{\rho}$ of the form $\R\ltimes N_0$, where $\R$ is a $1$-parameter group of $\R\times K \subset G_{\rho}$  with non-trivial projection to the $\R$-factor.   
\end{cor}
\begin{proof}
    The $1$-parameter group is given by $c(t)$ in the proof of the previous proposition. 
\end{proof}

\section{Non-standard phenomena}\label{Section: Non Standard}

In this section, we try to understand the non-standard compact quotients $\Gamma \backslash X_{\rho}$ of plane waves $X_{\rho}=G_{\rho}/I$. In this case, the fundamental group $\Gamma$ is necessarily straight: in Example \ref{Example: non standard C-W}, we give an explicit example of such a situation. 
The main motivation here is to see whether $\Gamma$ is standard up to embedding $G_{\rho}$ in some bigger group.
\bigskip\,

Let $\Gamma$ be a straight discrete subgroup of $G_{\rho}$ acting properly and cocompactly on $X_{\rho}$, then up to finite index, $\Gamma = \langle \hat{\gamma} \rangle \ltimes \Gamma_0$, with $\Gamma_0:=\Gamma \cap \Heis$. Let $N_0$ be the Malcev closure of $\Gamma_0$ in $\Heis$. There are two cases:\\\\
\textbf{Case 1:}  $\hat{\gamma}$ is not contained in a $1$-parameter group of $G_{\rho}$.\\
Non-existence of a $1$-parameter group containing $\hat{\gamma}$ does not imply non-standard\-ness. This is already suggested by Proposition \ref{Proposition : semi/standard}. Indeed, it is proved there that in the standard case, $n_0^{-1} \hat{\gamma}$ is contained in a $1$-parameter group of $G_{\rho}$ for some $n_0 \in N_0$, but not a priori $\hat{\gamma}$. Here, we give an example where this occurs.

\begin{example}
Let $\widetilde{\Euc}_2 :=\R\ltimes \R^{2}$,  where $\R$ acts by rotations on $\R^2$. Let  $G_{\rho}=(\R \times \SO(2))\ltimes \Heis_5$ be the special polarized Poincar\'e group in dimension $4$ (recall that this is the subgroup of $\Isom(\Mink^{1,3})$ preserving a lightlike vector). Let $c(t)=(t, e^{it})$ be a $1$-parameter group in $\R \times \SO(2)$. Since it preserves $A^- \simeq \R^2$, one can consider $S_1:= c(t) \ltimes A^-$, which is isomorphic to $\widetilde{\Euc}_2$. 
Define $S:= S_1 \times \R$, where $\R$ is the center of $\Heis_5$. Clearly $S$ acts properly cocompactly on $X_{\rho}=\Mink^{1,3}$.
Moreover, $S$ contains straight lattices which are not contained in the image of the exponential map. Namely, take any lattice $\Gamma_1$ in the abelian subgroup $2\pi\Z \times \R^2 \subset \widetilde{\Euc}_2$ which does not intersect the $2 \pi \Z$-factor, and consider $\Gamma= \Gamma_1 \times \Z$. \\
\end{example}
\noindent\textbf{Case 2:} $\hat{\gamma}$ is contained in a $1$-parameter group of $G_{\rho}$. \\\\
When $\Gamma$ is non-standard, we ask the following question:\\

\textbf{Question:}
Can we find a compact group $C$ acting on $\Heis$, and construct an embedding $G\hookrightarrow \widehat G := (\R \times C \times K) \ltimes \Heis$, and a $G$-equivariant embedding $X\hookrightarrow \widehat X=\widehat G/ \widehat I$ into a homogeneous space,  such that the quotient $\Gamma\backslash \widehat X$ is standard? Here, $\widehat{I}:= (C' \times K) \ltimes A^+$, where $C'$ is a closed subgroup of $C$ preserving $A^+$. In particular, we obtain the semi-standard quotient $\Gamma\backslash X$ as a submanifold of the standard quotient $\Gamma\backslash \widehat X$. 

The answer to this question is provided in  the following theorem.

\begin{theorem}\label{Theorem: non-standard phenomena}
Let $G_{\rho}= (\R \times K) \ltimes_{\rho} \Heis$, $I=K \ltimes A^+$, and $X_{\rho}= G_{\rho}/ I$. Let $\Gamma=\langle \hat{\gamma} \rangle \ltimes \Gamma_0$ be a non-standard straight discrete subgroup acting properly and cocompactly on $X_{\rho}$.
If $\hat{\gamma}$ is contained in a $1$-parameter group, there is a $\widehat G$-homogeneous space $\widehat X$, which is a torus bundle over $X_{\rho}$, such that $\Gamma$ lifts to a standard discrete subgroup of $\widehat G$. 
\end{theorem}

In order to explain the construction of $\widehat X$ let us look first at the following non-standard quotient of a Cahen-Wallach space:

\begin{example}[\bf A non-standard Example]\label{Example: non standard C-W}
Let $K := \SO(2)$ act on $\Heis_5 = \R \times \C^2$ trivially on the center and by the standard representation on $A^+$ and $A^-$.  Furthermore, let $\R$ act on $\Heis_5$  by  $t \cdot (v,z_1,z_2)=(v,e^{it} z_1, e^{it} z_2)$ and consider $G_{\rho}= (\R \times K)\ltimes_{\rho} \Heis_5$.
Take $\alpha\in\R^*$ and define a three-dimensional subgroup  $N:=\R\cdot A'$, where $A':=\Span_{\R}\{c_1:=(1,\alpha i),\, c_2:=(-\alpha i,1)\}\subset \Heis_5$.  Then $N$ acts properly on $X_{\rho}=G_{\rho} / (K \ltimes A^+)$ and cocompactly on the $\Heis$-leaves.
Take a lattice $\Lambda$ in it and define $\Gamma := \langle \hat{\gamma} \rangle \times \Lambda$, with $\hat{\gamma}=(2 \pi, 1, 0)$. Then $\Gamma$ is a discrete subgroup of $G_{\rho}$ acting properly and cocompactly on $X_{\rho}$. Indeed, 
by \cite[Proposition 4.8]{KO} it suffices to show that $e^{it}A'\cap A^+=\{0\}$ for all $t\in\R$. Since $A^+=\R^2$, this is satisfied if and only if the linear equation $ \Im (r_2 e^{it}c_1+ r_2  e^{it}c_2)=0$ for $r_1,r_2\in\R$ admits only the trivial solution. Thus the assertion is equivalent to  $\det( \Im (e^{it}c_1,e^{it}c_2))=1+(\alpha^2-1)\cos^2t\not=0$ for all $t$, which is obviously satisfied.  
We want to show that it is non-standard for $\alpha \neq \pm1$. By Proposition \ref{Proposition : semi/standard}, it suffices to show that $A'$ is not invariant under ${\bf L}+\phi$ for any $\phi\in{\mathfrak k}$. Since ${\bf L}|_A$ equals multiplication by $i$ and $K=\SO(2)$ preserves $A'$, the condition $({\bf L}+\phi)A'\subset A'$ would give $\Span_{\R}\{c_1,\, c_2\}=\Span_{\R}\{i c_1,\, ic_2\}$, which implies $\alpha=\pm1$.
\end{example}

\paragraph{ \textbf{How to make Example \ref{Example: non standard C-W} standard}}  If one defines $\widehat{G} := (\R \times \mathbb{S}^1 \times \SO(2)) \ltimes \Heis_5$, where $\R$ acts like before, and $\mathbb{S}^1$ acts like $\R / \ker \rho$, then the diagonal $1$-parameter group $ D:=(t, e^{-i t})$ in $\R \times \mathbb{S}^1$ acts trivially on $\Heis_5$, hence preserves $N$. The  new homogeneous space $\widehat{X}=\widehat{G} /( \SO(2) \ltimes A^+)$ is a plane wave which is an $\mathbb{S}^1$-bundle over $X_{\rho}$. We have a natural embedding by a diagonal morphism map $d: G \to \widehat{G}$, $d(t,k,h)=(t,0,k,h)$, which induces an embedding $X \to \widehat{X}$, such that $\widehat{\Gamma} := d(\Gamma)$ is standard in $\widehat{G}$. Its syndetic hull $ D\times N$ in $\widehat{G}$ does not act transitively on $\widehat{X}$.

\textbf{Some comments :}
	\begin{enumerate}
		\item The compact $\mathbb{S}^1$-factor added here does not preserve $A^+$, this is why it cannot be added to the isotropy. So  the dimension here is increased.  
		\item  In this example, $\R$ acts through an elliptic matrix, but in general there may be a hyperbolic and a unipotent part. The problem is formulated in the sequel for general homogeneous plane waves. \\
	\end{enumerate}
 
\paragraph{\textbf{Modifying the group to have standard quotients}} 
Let $G_{\rho}= (\R \times K) \ltimes_{\rho} \Heis$, $I=K \ltimes A^+$, and $X_{\rho}= G_{\rho}/ I$. Let $\Gamma=\langle \hat{\gamma} \rangle \ltimes \Gamma_0$ be a non-standard discrete subgroup acting properly and cocompactly on $X_{\rho}$.

Assume that $\hat{\gamma}$ is contained in some $1$-parameter group  $g(t)$ of $G_\rho$. We identify the image of $\rho$ with a subgroup of  ${\sf Sp}_{2n}(\R)$   and consider the Jordan decomposition  $l(t):=\rho(r(g(t)))= d(t) \circ u(t)$, where $d(t)$ is the semi-simple part and $u(t)$ the unipotent part. Write also $d(t) = e(t) \circ h(t)$, where $e(t)$ is the elliptic part (with complex eigenvalues), and $h(t)$ the hyperbolic part (with real eigenvalues).

\begin{lemma}\label{Lemma: elliptic part of l(t)}
Let  $l(t)$ be a $1$-parameter group in { ${\sf Sp}_{2n}(\R)\subset \Aut(\Heis)$}, and $ l(t) = e(t) h(t) u(t)$ be as above. 
Assume that $l(1)$ preserves a subgroup $N$ of $\Heis$. Then $h(t)$ and $u(t)$ preserve $N$ for any $t$.
\end{lemma}
\begin{proof}
It follows from elementary linear algebra that $l(1)$ preserves $N$ if and only if $e(1)$, $h(1)$ and $u(1)$ do.
Moreover, if $h(1)$ preserves $N$, then $h(t)$ preserves $N$ for any $t$. Similarly, 
if $u(1)$ preserves $N$, then $u(t)$ preserves $N$ for any $t$.
\end{proof}

Let $N$ be the Malcev closure of $\Gamma_0$ in $\Heis$. Then $\rho(\hat{\gamma}) =l(1)$ preserves $\Gamma_0$, hence also $N$. The elliptic part $e(t)$ of  $l(t)$ is a $1$-parameter subgroup of $\Aut(\Heis)$, whose closure is a torus $\T^d$ in $\Aut(\Heis)$. It follows from Lemma \ref{Lemma: elliptic part of l(t)} above that the $1$-parameter group $D:=(t, e(t)^{-1})$ in $\R \times \T^d$ preserves $N$. 
As in the previous example, one can add the $\T^d$-factor to the group and define $\widehat{G}= (\R  \times \T^d \times K) \ltimes \Heis$. Define a new homogeneous space $\widehat{X}=\widehat{G}/\widehat{I}$,  where the isotropy is given by  $\widehat{I} = (C \times K) \ltimes A^+$, with  $C$ any compact subgroup of $\T^d$ preserving $A^+$. 
We have a natural injective morphism $d: G_{\rho} \to \widehat{G}$, $d(t,k,h)=(t, 0, k,h)$ inducing an embedding $\bar{d}: X_{\rho} \to \widehat{X}$.
Then $\widehat{\Gamma} := d(\Gamma)$ is standard in $\widehat{G}$, and  $D \ltimes N$ (note that the $D$-action on $N$ here is generically non-trivial) is a syndetic hull of $\hat \Gamma$ in $\widehat{G}$.
\bigskip

In general, one cannot expect that the extension $\widehat{X}$ admits a $\widehat{G}$-invariant pseudo-Riemannian metric, even in the case $\widehat{I}=I=K \ltimes A^+$. Let us take a closer look at the latter situation. The following proposition characterizes the spaces $\widehat{X}$
that are indeed extensions in a geometric sense.

\begin{proposition}\label{Prop6.5}
The extension $\widehat{X}$ admits a left $\widehat G$-invariant pseudo-Riemannian metric if and only if the normalizer of  $\a^+ \oplus \z$ has codimension one in $\hat \g$. 
In the latter case, $\widehat{X}$ admits a Lorentzian metric such that $\bar d: X_\rho \hookrightarrow \widehat{X}$ is an isometric embedding.
\end{proposition}

\begin{proof}
Denote by $p$ the projection $\hat\g \to  \q:=\hat\g/\i$. 
For $x\in\hat\g$ we abbreviate $p(x)$ as $\bar x$. The adjoint representation of $\i$ induces a representation $\varphi$ of $\i$ on $\q$. Recall that $\widehat{G}$ preserves a pseudo-Riemannian metric on $\widehat{G}/I$ if and only if $\q$ admits a $\varphi$-invariant scalar product $h$. The Lorentzian metric on $X_\rho$ induces a $\varphi$-invariant scalar product $\ip$ on $p(\g_\rho)$. 
Fix some $0 \neq z \in \z$. Then $\langle \bar z, \bar z \rangle=\langle \varphi_a(\overline{a'}),\bar z \rangle=-\langle \overline{a'},\varphi_a(\bar z)\rangle=0$, for suitable $a\in\a^+,a'\in\a^-$. Next, let $a' \in \a^-$ and choose $a \in \a^+$ such that $[a,a'] =z$.  Then $\langle \overline{a'},\bar{ z} \rangle=\langle \overline{a'}, \varphi_a(\overline{a'}) \rangle=0$. Thus, $\langle p(\a^-\oplus\z),\bar z \rangle=0$.
Now, take $0\not=a\in\a^+$, and choose $a'\in\a^-$ such that $[a,a']=z$ for some fixed $0\not=z\in\z$. Then $\langle \varphi_a (\bar{\bf L}),\overline{a'}\rangle=-\langle \bar{\bf L},\varphi_a(\overline{a'})\rangle =\langle \bar{\bf L},\bar z\rangle\not=0$. Thus $a \mapsto \varphi_a(\bar{\bf L})$ induces a bijective map from $\a^+$ to $p(\a^- \oplus \z)/p(\z)$, since $\langle p(\a^-),\bar z\rangle=0$. In particular, 
\begin{equation}\label{y}
   \varphi_{\a^+}(\overline{{\bf L}})  \equiv p(\a^-) \mod p(\z).
\end{equation} 
Suppose $h$ exists. The same reasoning above gives $h(p(\a^-\oplus\z),\bar z)=0$. 
Assume that $h(\bar{\bf L},\bar z)=0$. Using  (\ref{y}) and taking into account the $\varphi$-invariance of $h$, we obtain that $p(\a^-)\perp \q$, which is a contradiction. Hence $h(\bar{\bf L},\bar z)\not=0$. 
Now, consider $u \in \hat{\mathfrak{g}}$, $0 \neq a \in \a^+$ and $0 \neq a' \in \a^-$. We write $[a,a']=\omega(a,a')z$. Then $h(\varphi_a(\overline{u}), a') =-h(\overline{u}, \varphi_a(\overline{a'}))=-\omega(a,a')h(\overline{u},\overline{z})$. Hence
\begin{align}\label{Eq-h}
%	h(\overline{[u,a]}, \overline{a'}) 
 h(\varphi_a(\overline{u}), \overline{a'})=-\omega(a,a')h(\overline{u},\overline{z}).
\end{align}
Set $u={\bf L}$ in (\ref{Eq-h}). Since  (\ref{y}) holds,  $\overline{a'}^\perp \cap\, p(\a^-)$  has codimension one in $p(\a^-)$. Thus,  
$h$ is non-degenerate on $p(\a^-)$. 
Then (\ref{Eq-h}) implies that $ \bar{u} \in \bar{z}^\perp$ if and only if $u$ belongs to the normalizer of $\a^+ \oplus \z$. Indeed, if $h(\overline{u},\overline{z})=0$, then  $\varphi_a(\overline{u}) \in p(\heis) \cap   p(\a^-)^\perp = p(\z)$. Since $[u,a] \in \heis$, this yields $[u,a] \in \a^+ \oplus \z$. Conversely, if we take $a'$ such that $[a,a'] \neq 0$, then  (\ref{Eq-h}) implies $h(\overline{u},\overline{z})=0$. Since $\bar z^\perp$ has codimension one in $\q$, the normalizer of $\a^+\oplus\z$ must have codimension one in $\hat\g$. 
For the converse, take the restriction of $h$ to $p(\mathfrak{g}_\rho)$ equal to $\ip$. One easily checks that this can be extended to a $\varphi$-invariant Lorentzian scalar product on $\q$.
\end{proof}

\section{Equicontinuity of the parallel flow}\label{Section: Equicontinuity}

Let $(M,V)$ be a compact locally homogeneous plane wave. Since $M$ is complete, it is a quotient $\Gamma \backslash G_{\rho}/I$ by some subgroup $\Gamma$ of $G_{\rho}$ acting properly discontinuously on $G_{\rho} /I$. In this case, the action of the flow of $V$ is given by the $Z$-action, where $Z$ is the center of Heisenberg (Fact \ref{Fact: parallel flow is Z}).  Recall that the flow of  $V$ is said to be equicontinuous if it is relatively compact, considered as a one-parameter subgroup of the isometry group endowed with its Lie group topology.  Equivalently, the flow is equicontinuous if it preserves a Riemannian metric. The key point for this equivalence is the fact that the isometry group of a Lorentzian metric $g$ on $M$ is closed in $\mathsf{Homeo}(M)$  with respect to the compact-open topology (which in fact coincides on $\Isom(M,g)$ with the Lie group topology), see  \cite{nomizu}.
\subsection{Equicontinuity} 
\begin{theorem}\label{Theorem: equicontinuity}
Let $(M,V)$ be a compact locally homogeneous plane wave. The action of the parallel $V$-flow is equicontinuous.
\end{theorem}
\begin{proof}
Let $X$ be the universal cover of $M$, we have $M=\Gamma\backslash X$, where $\Gamma\subset \isom(X)$ is a subgroup acting freely, properly discontinuously and cocompactly on $X$. In the non-flat case, $\isom^{\circ}(X)$  is isomorphic to $G_{\rho}$ for some $\rho$. In the flat case, since $\Gamma$ preserves the lift of $V$ to $X$, we can reduce to the subgroup of the Poincar\'e group preserving a lightlike vector, namely, $\mathsf{SPol}$ which is also isomorphic to $G_{\rho}$ for some $\rho$ (see Definition \ref{Section: Flat case}).  And we have $\mathsf{Isom}{(M)}= \Gamma\backslash N_{G_{\rho}}(\Gamma)$. Moreover, by Fact \ref{Fact: parallel flow is Z}, the flow of $V$ on $M$ is given by the left $Z$-action on the double quotient $\Gamma\backslash G_{\rho}/I$ (the action is well defined, since $\Gamma$ centralizes $Z$). The equicontinuity of the flow of $V$ is equivalent to $\pi(Z)$ being relatively compact in $\Gamma\backslash N_{G_{\rho}}(\Gamma)$, where $\pi: N_{G_{\rho}}(\Gamma) \to \Gamma \backslash N_{G_{\rho}}(\Gamma)$ is the natural projection. So to show that the action is equicontinuous, it is enough to show that $\Gamma\backslash\overline{\Gamma Z}$ (here the closure is in $N_{G_{\rho}}(\Gamma)$) is compact. 
There are two cases:\\
\textbf{1)} Straight case: We know from Theorem \ref{Theorem: standarness in straight case} that $\Gamma$ contains a finite index subgroup $\Gamma'$, which is a cocompact lattice of a closed subgroup $\Z\ltimes N_0$ of $G_{\rho}$, with $N_0$ a closed subgroup containing the center $Z$. On the one hand,  $\overline{Z\Gamma'}\subset \Z\ltimes N_0$. However, the quotient $\Gamma'\backslash(\Z\ltimes N_0)$ is compact. This implies that the $V$-flow is equicontinuous on $\Gamma' \backslash X$, i.e. it preserves a Riemannian metric $h$ on $X$ which is $\Gamma'$-invariant. Now, averaging $h$ over the representatives of the (finite) quotient space $\Gamma'\backslash\Gamma$ defines a new Riemannian metric $h^* :=\sum_{i=1}^{r} \gamma_i^*(h)$, where $\{\gamma_1,\dots,\gamma_r\}$ is a set of such representatives, which is $\Gamma$-invariant and preserved by the flow of~$V$. \\
\textbf{2)} Non-straight case: Similarly, $\Gamma$ contains a finite index subgroup $\Gamma'$ which is a cocompact lattice in some closed Lie subgroup $N$ that contains the center $Z$ (Theorem \ref{Theorem: standarness in non-straight case}). This implies in the same way as in the straight case that the $V$-flow preserves a Riemannian metric on $M$.
\end{proof}

Now we can prove all statements in Theorem \ref{Introduction-Theorem: parallel fields}.
\begin{proof}[Proof of Theorem \ref{Introduction-Theorem: parallel fields}]
The existence of a parallel vector field $V$ and the last statement (existence of a parallel line field) are proved in Theorem \ref{Theorem section flat: parallel fields}. When $V$ is timelike, the isometry group preserving $V$ is $\SO(n) \ltimes \R^{1+n}$, so the claim follows from Bieberbach Theorem.  The case where $V$ is lightlike follows from Theorem \ref{Theorem: equicontinuity}. For $V$ spacelike, examples of non-equicontinuous Anosov as well as partially hyperbolic flows are given in Example \ref{Example: SOL-Anosov}.
\end{proof}

\subsection{Equicontinuity from standardness}
In the standard case, we can give a more general proof of equicontinuity of the lightlike parallel flow.
In this case, $X_{\rho}$ identifies with a Lie group $S$ endowed with a left-invariant Lorentzian metric, and $M$ is finitely covered by $M':=\Gamma \backslash S$, where $\Gamma$ is a uniform torsion free discrete subgroup of $S$. The action of the $V$-flow on $X_{\rho}$ induces an action on $S$ which commutes with all left translations. This defines a left invariant parallel lightlike vector field on $S$. So equicontinuity of the $V$-flow on $M$ is equivalent to the equicontinuity of the induced vector field on $\Gamma \backslash S$. The latter is a consequence of Theorem \ref{Theorem: equicontinuity on Lie group} below.

\begin{theorem}\label{Theorem: equicontinuity on Lie group}
Let $G$ be a Lie group with a left invariant Lorentzian metric. Let $V$ be a left-invariant parallel lightlike vector field on $G$. Then $\ad_V$ is skew-symmetric for a Riemannian scalar product on $\g$.

\end{theorem}
\begin{proof}
Let $\langle \cdot\,,\cdot\rangle$ denote the metric on $G$ and also the induced scalar product on~$\g$.
The vector field $V$ is left invariant, so its flow corresponds to the right action of a one parameter group, say $v^t$ of $G$. 
Since the Lorentzian metric on $G$ is left invariant, left multiplication by $v^t$ is isometric. 
On the other hand, $V$ is parallel, hence Killing. Thus, right multiplication by $v^t$ is also isometric. In particular, conjugacy by $v^t$ is isometric, hence $\ad_V: \mathfrak{g}  \to \mathfrak{g}$  is skew symmetric  with respect to $\langle \cdot\,,\cdot\rangle$.

Assume that $\langle \cdot\,,\cdot\rangle$ restricted to $\q:=\ker \ad_V$ is degenerate. Since $V\in\q$ is lightlike, we have $\q=\R\cdot V \oplus \q_1$, where $\q_1$ is non-degenerate. Then $\q_0:=\q_1^\perp\subset\g$ is invariant under $\ad_V$. So we obtain a skew-symmetric linear map $A:=\ad_V|_{\q_0}$ on the Lorentzian space $\q_0$. By construction, $\ker A=\R\cdot V$. Since $V$ is lightlike, also $V\in(\ker A)^\perp=\mathsf{im}(A)$ holds. So we can choose $e_1\in\q_0$ such that $A e_1=V$. Then $\langle e_1, V\rangle=0$, which implies that $e_1$ is spacelike. Moreover, $\langle e_1, V\rangle=0$ implies $e_1\in (\ker A)^\perp=\mathsf{im}(A)$, thus $e_1=A(e_2)$. Now we compute $\langle V,[e_1,e_2]\rangle=\langle Ae_1,[e_1,e_2]\rangle=-\langle e_1,A[e_1,e_2]\rangle=-\langle e_1,[Ae_1,e_2]\rangle-\langle e_1,[e_1,Ae_2]\rangle =-\langle e_1, Ae_2\rangle=-\langle e_1,e_1\rangle >0$. On the other hand, by the Koszul formula, $[\g,\g]^\perp$ consists of all elements $Y\in\g$ for which $\nabla Y:\g\to\g$ is symmetric. Hence $V$ belongs to $[\g,\g]^\perp$, which gives a contradiction.  Consequently, $\q$ is non-degenerate, thus Lorentzian. This implies $\g=\q\oplus \q^\perp$. We define a Riemannian product $(\cdot\,,\cdot)$ on $\g$ by $(\cdot\,,\cdot)=\langle\cdot\,,\cdot\rangle|_{\q^\perp}\oplus (\cdot\,,\cdot)'$, 
where $(\cdot\,,\cdot)'$ is any Riemannian scalar product on $\q$. Then $\ad_V$ is skew-symmetric with respect to $(\cdot\,,\cdot)$.
\end{proof}

\section{On standardness of more general locally homogeneous structures}\label{Section: More homogeneous structures}
In the study of standardness of compact quotients $\Gamma \backslash X_{\rho}$, we did not use the Lorentzian nature of the homogeneous spaces $X_\rho=G_{\rho}/I$ provided by the restrictions on the $\rho$-action. In this section, we give an analogous statement to that in Sections \ref{Section: Fundamental group} and \ref{Section: syndetic hull, standard semi-Standrad}. We consider $G_{\rho} := (\R \times K) \ltimes_{\rho} \Heis$, as defined in \ref{Eq: G_rho} where here $\rho$ is a general action. Moreover, we drop the conditions on $K$ except that it is compact and connected, and we consider also a potentially smaller isotropy group $I=C\ltimes A^+$, where $C$ is a closed subgroup of $K$ that preserves $A^+$. We consider the homogeneous space $X_{\rho}:=G_{\rho}/I$.

 The proof of Theorem \ref{Theorem Gamma} is based on the existence of automorphisms $\Psi_a$ of $G_\rho$ which are the identity on $\R\times K$ and restricted to $\Heis$ of the form 
\begin{equation}\label{dilation}h_a:\ \Heis \longrightarrow \Heis, \ (z, \xi)\longmapsto(a^2 z, a \xi) 
\end{equation} 
for $a \in \R\setminus\{0\}.$  The next lemma immediately implies that such automorphisms also exist for the more general groups $G_\rho$ that are considered in this section (up to replacing $G_\rho$ by an isomorphic group $G_{\rho'}$). Automorphisms of $\Heis$ of the form \eqref{dilation} are called homotheties (or dilations). 
\begin{lemma}
  Let $G_\rho$ be as defined above. Then there exists a homomorphism $\rho':\R\times K\to\Aut(\Heis)$ such that $G_\rho$ and $G_{\rho'}$ are isomorphic and $\rho'(\cdot)h_a=h_a\rho'(\cdot)$ for any homothety $h_a$ of $\Heis$.   
\end{lemma} 
\begin{proof}
Since $K$ is compact, we may assume that the action of $K$ preserves the two factors of $\Heis=\R\times \C^n$. Then $K$ acts trivially on the center $\R$. On $\C^n$ it acts by symplectic morphisms with respect to the symplectic form $\omega$ that defines $\Heis$ as an extension of $\C^n$ by $\R$. There exists a $K$-invariant positive definite scalar product $B$ on $\R \times \C^n$ such that $\R\perp\C^n$. Let $\k$ be the Lie algebra of $K$. We decompose $\C^n=V_0\oplus V_1$, where $V_0:=\{v \in V, \  k(v)=0, \ \text{for all} \  k\in \k\}$ and $V_1$ is equal to $(\R\times V_0)^\perp$ with respect to $B$. Note that $V_1$ is spanned by $k(v)$ for arbitrary $k\in\k$ and $v\in V$. In particular, $\omega(V_0,V_1)=0$, thus $\omega$ is non-degenerate on $V_0$. Hence $\R\times V_0$ is a Heisenberg group. Each derivation ${\bf L}$ of $\heis$ which commutes with the action of $\k$ satisfies ${\bf L}(V_0)\subset \R\times V_0$ and ${\bf L}(V_1)\subset V_1$. Since $\R\times V_0$ is a Heisenberg group, there exists an element $h\in V_0$ such that ${\bf L}'={\bf L}+\ad(h)$ maps $V_0$ to $V_0$. Since $h$ is in $V_0$, ${\bf L}'$ commutes with the action of $\k$. Thus we may replace the derivation ${\bf L}$ that generates the action of $\R$ on $\Heis$ by ${\bf L}'$  to obtain $\rho'$. By construction $\rho'(\R\times K)$ preserves $\R$ and $\C^n$, which implies the assertion.   
    \end{proof}

\begin{theorem}\label{Theorem: General homogeneous}
      Let $G_\rho$ be defined as in the beginning of this section and let $\Gamma$ be a discrete subgroup of $G_{\rho}$ acting properly cocompactly on $X_{\rho}$. Then
    \begin{itemize}
        \item[(1)] $\Gamma$ is virtually nilpotent, and $X_{\rho}$ is standard,
        \item[(2)] or $\Gamma\cong \Z\ltimes \Gamma_0$, where $\Gamma_0$ is virtually nilpotent, and $X_{\rho}$ is semi-standard.
    \end{itemize}
Moreover, when $C=K$, the syndetic hull of $\Gamma$ (resp. $\Gamma_0$) acts transitively on $X_{\rho}$ (resp. on the $K \ltimes \Heis$-leaves).
\end{theorem}
 
\begin{proof}
    We proceed analogously to the proofs of Theorem \ref{Theorem: standarness in non-straight case} and Theorem \ref{Theorem: standarness in straight case}. Either $\Gamma$ is non-straight or straight. In the non-straight case, since the representation $\rho$ commutes with $\Heis$ homotheties, the same proof as in Theorem $\ref{Theorem Gamma}$ allows to get that $\Gamma$ is generated by elements in a strong Zassenhaus neighborhood. Hence the existence of a (nilpotent) syndetic hull by Fact \ref{Thurston: strongly Zassenhaus neighborhood}, and the standardness of $\Gamma \backslash X_{\rho}$. In the straight case, we get that $\Gamma\cong \Z\ltimes \Gamma_0$, where $\Gamma_0 \subset K\ltimes \Heis$. Here, we discuss two cases: either $\Gamma_{0}$ has a finite projection to $K$, or the projection is not discrete. In the first case, $\Gamma_0 \cap \Heis$ has finite index in $\Gamma_0$, and its Malcev closure in $\Heis$ acts properly and cocompactly on the $\Heis$-leaves. In the second case, we proceed as in the non-straight case, and show that $\Gamma_0$ is generated by elements in  a strong Zassenhaus neighborhood of $K \ltimes \Heis$, hence the existence of a (nilpotent) syndetic hull of $\Gamma_0$ in $K \ltimes \Heis$ and the semi-standardness of $\Gamma \backslash X_{\rho}$. Now, when $C=K$,  the transitive action of the syndetic hull of $\Gamma$ (resp. $\Gamma_0$) on $X_{\rho}$ (resp. on a $\mathcal{F}$-leaf) follows from Proposition \ref{Proposition-Appendix A: transitive action under existence of a lattice} and Proposition \ref{Proposition-Appendix A: G_rho linear}.
\end{proof}

\begin{example}[Non-transitive action]
    In the modification of Example \ref{Example: non standard C-W}, $C=\SO(2)$ and $K=\mathbb{S}^1\times \SO(2)$. The syndetic hull $H$ of $\Gamma$ does not act transitively on $\widehat{X}$. Indeed, $H$ acts freely on $\widehat{X}$ and we have $H\backslash \widehat{X}\cong \mathbb{S}^{1}$.
\end{example}
\begin{remark}[Carnot groups]
     The Heisenberg group is a particular case of a Carnot group, to which the notion of a homothety can be generalised. An analogous version of Theorem \ref{Theorem: General homogeneous} holds for discrete subgroups of $(\R\times K)\ltimes_{\rho} N$, where $N$ is a Carnot group admitting a homothety that commutes with the action of $\R\times K$. 
\end{remark}

\begin{remark}
The homogeneous spaces in Theorem \ref{Theorem: General homogeneous} are not necessarily Lorentzian (or pseudo-Riemannian), as it appears from Proposition \ref{Prop6.5}. An easy situation that occurs here is when the $\rho$-action is trivial, i.e. $G_{\rho}=(\R\times K)\times \Heis$. Remember that $G_{\rho}$ preserves a Lorentzian metric on $G_{\rho}/I$ if and only if the $\ad(I)$-action on $\mathfrak{g}/\mathfrak{i}$ is an infinitesimal isometry of some Lorentzian scalar product. Here, $\ad(h)$, for $h \in \a^+$, is nilpotent of degree $2$, but a nilpotent non-zero endomorphism of a vector space is an infinitesimal isometry of some Lorentzian scalar product if and only if its nilpotency order equals $3$ (this is a well-known fact in pseudo-Riemannian geometry; for a proof see \cite[Lemma 5.7]{Content1}).
 
\end{remark}

\appendix

\section{Cocompact proper actions of Lie groups}\label{Appendix A}
In this appendix we consider cocompact proper actions of Lie groups admitting
a torsion free uniform lattice.
\begin{proposition}\label{Proposition-Appendix A: transitive action under existence of a lattice}
     Let $G$ be a connected Lie group which admits a torsion free uniform lattice $\Gamma$. Assume that $G$ acts properly cocompactly on a contractible manifold $X$.  
     Then $G$ acts transitively.
 
\end{proposition}
The proof uses techniques from cohomology theory of discrete groups. 
 
\begin{proof}
Since $\Gamma$ is torsion free, $\Gamma \backslash X$ is a closed $K(\Gamma,1)$ manifold. We have by  \cite[VIII, (8.1)]{brown2012cohomology} that the cohomological dimension of $\Gamma$ is equal to $\dim (\Gamma \backslash X)$. 
  Now, let $K$ be a maximal compact subgroup of $G$. We know that $G$ is diffeomorphic to $K\times\R^{k}$ see \cite{iwasawa1949some}. We claim that  $\Gamma\backslash G/K$ is a $K(\Gamma, 1)$ manifold. Indeed, $\Gamma$ acts freely (torsion free) and cocompactly. Moreover, $\Gamma$ acts properly, since the fiber bundle map $\pi: G\longrightarrow G/K$ is a proper map. Hence, 
  \cite{brown2012cohomology} we get that $\dim(\Gamma\backslash G/K)=\dim(\Gamma\backslash X)$, i.e. $\dim(G)=\dim(K)+\dim(X)$. Because the $G$-action is proper (in particular the stabilizer of any point is compact), we conclude that the $G$-orbits are open. The claim follows from the connectedness of~$X$. 
\end{proof}

When $G$ is linear, Selberg's Lemma applies. Namely, any finitely generated subgroup of $G$ 
is virtually torsion free. For linear groups we get the following corollary.
\begin{cor}
    Let $G$ be a connected linear Lie group which admits a uniform lattice $\Gamma$. Assume that $G$ acts properly cocompactly on a contractible manifold $X$.  
     Then $G$ acts transitively.
\end{cor}

\begin{proposition}\label{Proposition-Appendix A: G_rho linear}
$G_\rho$ is linear.
\end{proposition}
\begin{proof}
 We have a natural morphism $f: G_{\rho} \to \Aut(\Heis)\ltimes \Heis$ \; given by $f(r,k,h)= (\rho(r,k),h)$ (which is not necessarily injective). Moreover, the group  $\Aut(\Heis)\ltimes \Heis$ is linear. Indeed, it has a trivial center (due to existence of homotheties), hence the adjoint representation $\Ad:\Aut(\Heis)\ltimes \Heis \to \GL(\g)$, where $\g$ is the Lie algebra of $\Aut(\Heis)\ltimes \Heis$, is an embedding (faithful). Define now $\Phi: G_{\rho} \longrightarrow (\R\times K) \times \GL(\g) $, $\Phi(r,k,h) = (r,k, \Ad (f (r,k,h)))$. Then $\Phi$ is clearly a faithful morphism into $(\R\times K) \times \GL(\g)$, which is a linear group. The claim follows.   
\end{proof}

\section{A non-periodic example}\label{Appendix B}
Let $L=\R \ltimes \R^{n+1}$, with coordinates $(v,y) \in \R^{n+1}=\R \times \R^n$, and the $\R$-action defined by $t \cdot (v,y) = (v, R_t(y))$, where $R_t$ is some {periodic} elliptic action on $\R^n$. Define a Lorentzian left-invariant metric $g$ on $L$, such that the induced metric on $\R^{n+1}$ is degenerate, and $V:= \partial_v$ is lightlike. Then $(L,g)$ is a homogeneous plane wave, by \cite[Theorem 3]{Leis}. Let $\Gamma := \langle \hat{\gamma} \rangle \times \Gamma_0$, with $\Gamma_0$ a lattice in $\R^{n+1}$ and $\hat{\gamma}$ generates a lattice in the $\R$-factor acting trivially on $\Gamma_0$. Suppose further that $\Gamma_0$ does not intersect the subgroup generated by the $v$-translations. Then $\Gamma$ is a (uniform) lattice in $L$, and the flow of $V$ is not periodic in $\Gamma \backslash L$.

\end{document}